\documentclass[11pt,reqno]{amsart}
\usepackage{amsmath, amsthm, amsfonts, amssymb}
\usepackage{graphicx, amscd, color, bm, enumitem}    
\usepackage[colorlinks]{hyperref}
\usepackage[T2A]{fontenc}
\usepackage[utf8]{inputenc}
\usepackage[mathscr]{eucal}

\newtheorem{theorem}{Theorem}[section]
\newtheorem{corollary}[theorem]{Corollary}
\newtheorem{definition}[theorem]{Definition}

\newtheorem{lemma}[theorem]{Lemma}

\newtheorem{proposition}[theorem]{Proposition}
\newtheorem{remark}[theorem]{Remark}

\newcommand{\be}{{\bm e}}
\newcommand{\Ord}{\mathop{\rm Ord}\nolimits}

\newcommand{\bmb}{{\bm \beta}}
\newcommand{\bi}{{\bm i}}
\newcommand{\bj}{{\bm j}}
\newcommand{\IN}{{\subseteq}}

\newcommand{\ga}{{\gamma}}
\newcommand{\Ga}{{\Gamma}}
\newcommand{\al}{{\alpha}}
\newcommand{\hga}{{\hat\gamma}}
\newcommand{\Nfi}{{N_\phi}}
\newcommand{\eS}{{\mathcal S}}

\begin{document}
\title{Self-similar dendrites with finite boundary and $P$-sprouts}
\thanks{The study was carried out under the state contract of the Sobolev Institute of Mathematics (project FWNF-2026-0026).}
\date{}
\author[Andrei Tetenov, Ivan Yudin, Dmitrii Drozdov]{Andrei Tetenov, Ivan Yudin, Dmitrii Drozdov}
\begin{abstract}
Each self-similar dendrite $K$ with a finite self-similar boundary defines a finite acyclic edge-labeled bipartite graph $\Gamma$,  called the sprout of $K$. The paper shows that the sprout  $\Gamma$   determines the combinatorial properties of the dendrite $K$ and its topological structure.\\

\end{abstract}
\maketitle
\textbf{Mathematics Subject Classification.} Primary: 28A80, 54F50; Secondary: 05C20.\\
\textbf{Keywords and phrases:} self-similar set,  dendrite,   intersection graph,  single intersection property, intersection graph, ramification points, acyclic fractal set.

\section{Introduction}
Let $\eS=\{S_1,...,S_m\} $ be a system  of contractive injections of a complete metric space $X$.
A  compact non-empty set $K$ is called the attractor of a system $\mathcal{S}$ if it satisfies the equation
$K=S_1(K)\cup...\cup S_m(K)$.   The existence and uniqueness of the attractor $K$ was proved by J.Hutchinson \cite{Hut}. 
The subsets $K_{i_1...i_n}= S_{i_1}...S_{i_n}(K)$ are called  copies of $K$. 
The set of those points $x\in K$, whose images are the intersection points of different copies of $K$, is the self-similar boundary of $K$. 
If the attractor $K$ is connected and does not contain a closed curve, it  is called a self-similar dendrite. 
The main object of this paper is self-similar dendrites with a finite boundary. 
These dendrites possess the single intersection property (SIP) and therefore have a clear combinatorial characterization.\\

The study of dendrites dates back to papers \cite{Maz, Men, Mil, W} that appeared in the early 1920-s and defined key objectives in dendrite topology. The main work of that period is the doctoral dissertation \cite{W} of Tadeusz Vazhevsky. He gave a rigorous systematic description of the topology of dendrites and proved the existence of a unique locally connected continuum with the property of universality. In such a space, the ramification points have a strictly defined order, and any other dendrite with a similar or lesser ramification order can be homeomorphically embedded in it. Karl Menger was one of the main researchers of the theory of curves and dendrites in the 1920-s. In \cite{Men} he developed a general theory of curves. He also gave a classification of dendrite points and investigated their properties in the context of dimension theory. S.Mazurkevich laid the foundations for the study of locally connected continua in \cite{Maz}. In subsequent articles  he investigated the conditions under which a continuum is a dendrite and dealt with the issues of their embeddability. K. Zarankiewicz formulated the topological question of whether dendrites are "non-contractible", that is, not homeomorphic to their own subset. In 1932 \cite{Mil}  Miller refuted this hypothesis by constructing a counterexample known as the Miller dendrite, which is homeomorphic for its part.
An exhaustive overview of the research in this area  can be found in \cite{Char}. 
The inception of the theory of self-similar sets \cite{Hut} initiated the study of self-similar dendrites. In  1985, M. Hata  \cite{Hata85} proved that the set of endpoints of a non-trivial self-similar dendrite is infinite. In 1995, J. Kigami \cite{Kig,Kig95} investigated the shortest-path metric in post-critically finite self-similar dendrites and constructed regular Dirichlet forms for such dendrites. Under certain conditions, Julia sets of complex polynomials may be dendrites. 
Hubbard and Douady \cite{DH1984} described this case in detail. Such Julia sets arise if the critical point is strictly preperiodic. Hubbard trees were introduced to study their structure. This is a finite tree that connects all the points of the critical point orbit inside the Julia set. It contains all the combinatorial information about the Julia set, allowing the classification of polynomials.
As  follows from the results of A.Kameyama \cite{Kam}, if a self-similar dendrite is the attractor of some invertible system of contractions, it is homeomorphic to a Julia set of some polynomial.
 The article \cite{SSS2} by C. Bandt and K. Keller deserves special attention, as it contains an intersection graph criterion for a  self-similar continuum with the single intersection property being a dendrite (Proposition 9). The preprint \cite{SSS6} by C. Bandt and J. Stanke introduces the idea of a 'main tree' for post-critically finite self-similar dendrites.
The topology automaton approach to the study of self-similar dendrites was considered in \cite{Bandt2024, RZ2020, RWYZ2023}.
Self-similar dendrites with non-trivial intersection patterns and infinite ramification were studied in \cite{AST2024, AK2025}.
The growth of self-similar dendrites was considered in \cite{barua2022}.

A simple and intuitive way to construct self-similar dendrites in the plane and in space is provided by polygonal (respectively, polyhedral) systems.
This method was described in \cite{STV} and expanded in \cite{Adam}.

In this paper, we develop a method for defining self-similar dendrites with single point intersections using simple combinatorial schemes called sprouts. These schemes determine self-similar dendrites up to isomorphism. We provide an algorithm that enables the reconstruction of the main topological elements of a self-similar dendrite from its sprout.

The main theorems proved in the paper are as follows.

The isomorphism theorem for self-similar dendrites (Theorem \ref{isom}) states that  if systems ${\eS}=\{S_1,...,S_m\}$ and $\tilde {\eS}=\{\tilde S_1,...,\tilde S_m\}$ have isomorphic $P$-sprouts $\Gamma$ and $\tilde \Gamma$, then their attractors $K$ and $ \tilde K$ are isomorphic. 

Theorem \ref{inf_ram} states the conditions under which  the  set of addresses of a boundary point $p\in \partial K$ is finite, countable, or uncountable and gives a formula for finding  the set of addresses of the point $p$.

The third is  Theorem \ref{ord_b}   on the order of boundary points with respect to the main tree. \\ 

In Section 3 we introduce a $P$-sprout $\Gamma$ that defines a self-similar dendrite $K$ with a finite self-similar boundary and  define the index diagram of $\Gamma$.

In Section 4 we consider   SIP systems of continua and refinement sequences of SIP systems, which serve as a tool to prove  Theorem \ref{isom}).

In Section 5 we consider the main tree of a self-similar dendrite, evaluate the orders of the points of the attractor, and determine the cardinality of the set of addresses of  boundary points. The main statements of this section are Propositions \ref{ordx} and \ref{cpgam} and Theorem \ref{inf_ram}.

In section 6 we consider a semigroup $G_\phi $ of  transformations of the set $P$, and prove  Theorem \ref{ord_b} on the orders of boundary points and ramification points of the main tree. Theorem \ref{radr}
gives a method of finding the addresses of the ramification points of the main tree.

\section{Preliminaries}\label{prel}

\subsection{Self-similar sets}\label{sss} Let ${\eS}=\{S_1,  S_2, \ldots,  S_m\}$ be a system of injective contractions in a complete metric space $(X,  d)$. 
There is a unique compact non-empty set  $K{\subset} X$ such that $K =\bigcup\limits_{i = 1}^m S_i (K)$. The set $K=K({\eS})$  is called the attractor of the system  ${\eS}$ \cite{Hut}. The subsets $K_i=S_i(K), i\in\{1,...,m\}$, are called {\em copies} of the attractor $K$.

\subsection{The semigroup $G_{\eS}$}\label{GeS}  The set $I=\{1, 2, \ldots, m\}$ is called the {\em set of indices},  the set $I^*=\bigcup\limits_{n=0}^\infty I^n$ 
is the set of  {\em multiindices} ${\bf {j}}=j_1j_2\ldots j_n$. For a multiindex
${\bf {j}}\in I^n$, the map $S_{\bf {j}}=S_{j_1}\cdot S_{j_2}\cdot\ldots\cdot S_{j_n}$ defines a copy $K_{\bf {j}}=S_{\bf {j}}(K)$ of order $n$. The set $G_{\eS}=\{S_{\bf {i}},{\bf {i}}\in I^*\}$ is a semigroup, generated by the system ${\eS}$.\\

\subsection{Addresses and the index map}\label{ptoK} Since any sequence 
${\bf\alpha}=\alpha_1\alpha_2\ldots\in I^\infty$ specifies a unique point 
$\pi({\bf\alpha})=\bigcap\limits_{n=1}^\infty K_{\alpha_1\ldots\alpha_n}$ \cite{Hut},
we define a map
$\pi:I^{\infty}\rightarrow K$  called the {\em index map}.  If $\pi({\bf\alpha})=x$,  then ${\bf\alpha}$ is called {\em an address} of the point $x$.  
 We say that an address ${\bf\alpha}$ is preperiodic if ${\bf\alpha}={\bf {j}}\bar{\bf {i}}$ for some ${\bf {i}},{\bf {j}}\in I^*$.

A point $y_k \in K$ is called a \emph{$k$-th predecessor} of a point $y \in K$
if there is a word $\mathbf{i} \in I^k$ such that $y = S_{\mathbf{i}}(y_k)$.
For an address $\alpha = \alpha_1 \alpha_2 \ldots \alpha_n \ldots$ of the point $y$,
the points $y_k = \pi(\alpha_{k+1} \alpha_{k+2} \ldots)
= S_{\alpha_1 \alpha_2 \ldots \alpha_k}^{-1}(y)$, $k = 1, 2, \ldots$,
form the sequence of predecessors of $y$
corresponding to the address $\alpha$.

\subsection{Self-similar boundary and the SIP}\label{dK}The union ${\mathcal C}=\bigcup\limits_{i\neq j} (K_i\cap K_j)$ is called { \em the critical set} of ${\eS}$. 
The set  of all predecessors of points of ${\mathcal C}$, i.e. the set of all  $x\in K$ such that for some ${\bf {i}}=i_1\ldots i_n$, $S_{{\bf {i}}}(x)\in {\mathcal C}$, is called {\em the self-similar boundary} ${\partial} K$ of the attractor $K({\eS})$ \cite{Mor}.
For any $n\in \mathbb{N}$ we define ${\mathcal C}^{(n)} = \bigcup\limits_{|\bf{i}| = n-1} S_{\bf{i}} ({\mathcal C})$.

Two multiindices  ${\bf {i}},{\bf {j}}\in I^*$ are incomparable, if neither of them is an initial subword of the other.

For any incomparable multiindices ${\bf {i}},{\bf {j}}\in I^*$, $K_{\bf {i}}\cap K_{\bf {j}}=S_{\bf {i}}({\partial} K)\cap S_{\bf {j}}({\partial} K)$. \\
The self-similar set $K({\eS})$ {\em has a finite self-similar boundary} if the set ${\partial} K$ is finite. 
The set $K({\eS})$ has the {\em finite intersection property} (FIP) if all the sets $K_i\cap K_j, i\neq j$ are finite and it has the {\em single intersection property}   (SIP) if for any different $i,j\in I$, $\#(K_i\cap K_j)\le 1$.\\

\subsection{Dendrites and sprouts}\label{DandS}

We call the attractor $K$ a \emph{self-similar continuum} if it is connected. By Hata's Theorem~\cite{Hata85}, any self-similar continuum is locally connected.

A \emph{dendrite} is a locally connected continuum that does not contain a simple closed curve~\cite{Kur}; equivalently, it is an acyclic continuum. Consequently, every acyclic self-similar continuum is a self-similar dendrite.

It is well known that each dendrite is uniquely arcwise connected~\cite[Theorem~1.1(16)]{Char}. This property yields the following:

\begin{proposition}\label{minarc}
Let $D$ be a dendrite and $A, B \subset D$ be disjoint subcontinua. There is a unique arc $\gamma \subset D$ with endpoints $x \in A$, $y \in B$ such that $\gamma \cap A = \{x\}$ and $\gamma \cap B = \{y\}$. We call $\gamma$ the \emph{minimal arc} connecting $A$ and $B$.
\end{proposition}

The structure of self-similar dendrites with finite boundary is particularly simple. 
\begin{proposition}\label{dsip}\cite{TYK}
Suppose that $K$ is a self-similar dendrite with finite self-similar boundary $\partial K$.
 Then $K$ has the single intersection property.\end{proposition}
 
Consequently, the set $\bigcup_{i \in I} S_i(\partial K)$ subdivides $K$ into pieces whose closures are the copies $K_i$.

We consider the following two sets: 
 The set $W = \{K_i\}_{i \in I}$ of copies of $K$ and   the set $B = {\mathcal C} \cup \partial K$ of intersection points of copies of $K$ and of boundary points of $K$.\\

If the system ${\eS}$ has the finite intersection property, we define its {\em intersection graph} $\Gamma({\eS})$ \cite{TYK} as
an edge-labeled bipartite graph $(W, B; E)$ with parts $W$ and $B$,  in which an edge $ e = (K_i, p ) \in E$ 
if and only if $p \in K_i$. Here, a label assigned to the edge $e=(K_i,p)$  corresponds to the boundary point $p_j \in {\partial} K$, for which $S_i(p_j)=p$. 
This labeled intersection graph $\Gamma({\eS})$ will be called {\em the sprout} of the system ${\eS}$.\\ 

Together with the graph $\Gamma$, we consider an edge-labeled directed pseudograph ${{\mathcal G}}$, derived from $\Gamma$, which we  call {\em the index diagram} of the system ${\eS}$.  
It shows  the relationship between the boundary points of the attractor $K({\eS})$.
The set of vertices of this graph is ${\partial} K$. 
For $x_i,x_j\in {\partial} K$, there is a directed edge from  $x_i$ to $x_j$ with edge label $w_k$ iff there is  $k\in I$ such that $S_k(x_j)=x_i$. \\

\begin{definition}\label{isom1}
 The attractors $K({\eS})$, $K({\EuScript T})$  of  systems of injective contractions ${\eS}=\{S_1,\ldots,S_m\}$ and ${\EuScript T}=\{T_1,\ldots,T_m\}$   are isomorphic if there is a homeomorphism $\varphi:K({\eS}) \to K({\EuScript T})$ such that for any $x\in K({\eS})$ and $i\in I$, $\varphi(S_i(x))=T_i(\varphi(x))$.    \end{definition}
 
If the latter is fulfilled, the sprouts $\Gamma({\eS})$ and $\Gamma({\EuScript T})$ are also    equivalent.
On the other hand, we will show  that if  the sprouts $\Gamma({\eS})$ and $\Gamma({\EuScript T})$ are equivalent, then the  attractors $K({\eS})$ and $K({\EuScript T})$ are isomorphic.\\

This correspondence between the equivalence of sprouts $\Gamma({\eS})$ and $\Gamma({\EuScript T})$ and the isomorphism of the attractors $K({\eS})$ and $K({\EuScript T})$ is consistent with earlier studies. In particular,~\cite{Adam} investigates families of self-similar dendrites where deformations of the system preserve the sprout and yield isomorphic attractors.

\begin{figure}[h!]
    \centering
    \includegraphics[width=\textwidth]{ 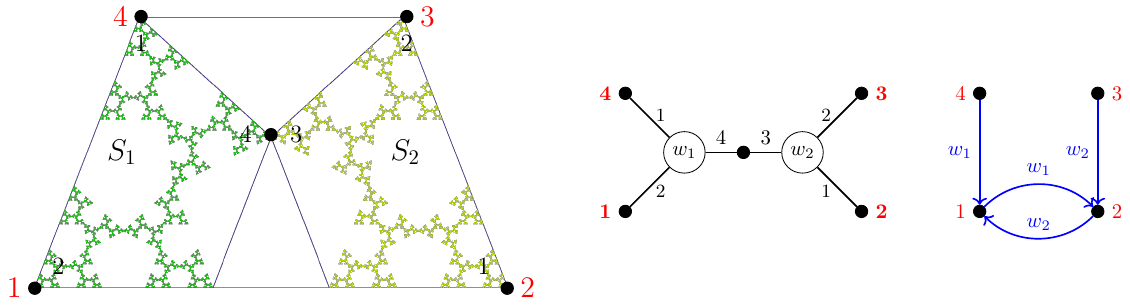} 
    {\Large$4 \mathop{\longrightarrow }\limits^{w_1}\ 1  \mathop{\longrightarrow }\limits^{w_1}\ 2  \mathop{\longrightarrow }\limits^{w_2}\ 1 \mathop{\longrightarrow }\limits^{w_1} 2\mathop{\longrightarrow }...$}
        \vspace{0.2cm}
    \caption{%
        A self-similar set $K(\mathcal{S})$ for $\mathcal{S}=\{S_1,S_2\}$, along with its sprout $\Gamma(\mathcal{S})$ (center) and index diagram $\mathcal{G}(\mathcal{S})$ (right). }
    \label{f1}
\end{figure}

On the Fig.\ref{f1}  the self-similar boundary $\partial K$ consists of four points $p_1,p_2$, $p_3,p_4$ (labeled $1,2,3,4$). The unique non-trivial intersection of the copies is $S_1(p_4)=S_2(p_3)$, which determines the edge structure of the sprout. In the index diagram, a directed edge with label $w_k$ from $x_i$ to $x_j$ signifies that $S_k(x_j)=x_i$. Consequently, any directed walk in $\mathcal{G}(\mathcal{S})$ encodes the sequence of preimages of its starting vertex, and the concatenation of edge labels along the walk yields the address of that vertex. The sequence below the diagrams corresponds to the infinite walk generated by iteratively taking preimages starting at $p_4$, and its label sequence $1121212\dots$ is the address of $p_4$.

The existence of a directed walk $w_{i_1}...w_{i_n}$ from $a$ to $b$ in this diagram implies that $S_{i_1...i_n}(b) = a$ and  $a \in K_{i_1...i_n}$. Thus, each infinite directed walk starting at the point $a$ defines an address of the point $a$. The number of addresses of a boundary point $a$ is equal to the number of infinite walks in
the index diagram that start from $a$, as shown on the right.

\section{Sprouts and index diagrams}\label{defspr}

\begin{definition}\label{defsp}
A $P$-sprout is a pair $(\Gamma, \varphi)$, where $\Gamma = (B, W, E)$ is a finite connected bipartite acyclic graph with parts $B$ and $W$, $P \subseteq B$, such that for any $b \in B\setminus P$, 
$\mathop{deg}_{\Gamma}(b) > 1$,
and $\varphi \colon E \to P$ is a map satisfying the following:
\begin{itemize} 
    \item[-] $\varphi(E) = P$ (surjectivity);
    \item[-]  for every $w \in W$, the restriction $\varphi|_{E(w)}$ is injective (local injectivity).
\end{itemize}
\end{definition}

\begin{definition}\label{defcs}
The {\em critical set} of a sprout $(\Gamma,\varphi)$ is defined as
$$C = \{b \in B : \deg(b) > 1\}.$$
For $b \in B$, denote $\widetilde{\varphi}(b)$ as the set $\varphi(E(b))$ of all $p \in P$ corresponding to edges incident to $b$. 

The {\em boundary} of the sprout $(\Gamma,\varphi)$ is defined as
$$\partial\Gamma = \bigcup_{n=1}^{\infty} \widetilde{\varphi}^n(C).$$

A sprout $(\Gamma,\varphi)$ is called {\em correctly defined} if $\partial\Gamma = P$. 

A correctly defined $P$-sprout is  {\em regular} if $\deg_\Gamma(w) > 1$ for every white vertex $w \in W$.
\end{definition}

\begin{remark}
In a regular $P$-sprout, for every   $v\in B\cup W$, each connected component of $\Gamma \setminus \{v\}$ contains at least one boundary point.
\end{remark}

\begin{figure}[htp]
\centering    
\includegraphics[width=\textwidth]{ 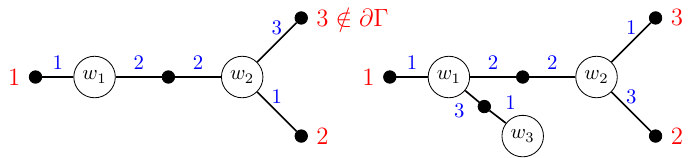}
\caption{ A sprout on the left is not  correctly defined ($3 
    \notin \bigcup\limits_{n=1}^{\infty} \widetilde{\varphi}^n(C)
    $), a sprout on the right is defined correctly  but is not regular.  }
\label{fig:cor_def}
\end{figure}

\begin{definition}
Two $P$-sprouts $(\Gamma, \varphi)$ and $(\Gamma', \varphi')$, where $\Gamma = (B, W, E)$, $\varphi \colon E \to P$, and $\Gamma' = (B', W', E')$, $\varphi' \colon E' \to P'$, are \emph{isomorphic} if there exists a bijection $f \colon B \cup W \to B' \cup W'$ such that

 $f$ preserves incidence: for every edge $e = \{u, v\} \in E$ with $u \in B$, $v \in W$, we have $\{f(u), f(v)\} \in E'$; 
 
   $f$ is compatible with the labelings: $\varphi'(f(e)) = f(\varphi(e))$ for all $e \in E$, where $f(e)$ denotes the edge $\{f(u), f(v)\}$.
\end{definition}

Let $E_P$ denote the set of edges of $\Gamma$ incident to vertices in $P$.

\begin{definition} \label{indiag}
The \emph{index diagram} of a sprout $(\Gamma, \varphi)$ is the labelled directed pseudograph $\mathcal{G}_P = (P, \mathcal{E}, \hat{\varphi})$ with vertex set $P$, constructed as follows. For every edge $e = (w_k, p_i) \in E_P$ such that $\varphi(e) = p_j$, there is a directed edge $\vec{e} = (p_i, p_j) \in \mathcal{E}$, and we set $\hat{\varphi}(\vec{e}) = w_k$.
\end{definition}

The index  diagram $\mathcal{G}_P$ is a pseudograph because  it may contain loops and it may contain multiple edges connecting some pair of vertices. Moreover, it need not be connected.

For any  $p\in P$  there is  $w\in W$ such that $(w,p)\in E_P$, therefore the outdegree of  each vertex $p$ in $\mathcal{G}_P$ is at least $1$.

 Since the graph $\Gamma$ is a tree, for any $p\in P$ and $w\in W$ there is at most one edge $e$ in $\Gamma$ incident with $p$ and $w$. Let $\varphi(e)=p'$. 
Then there is at most one edge $\vec e=(p,p')\in{{\mathcal E}}$ in ${{\mathcal G}}_P$ such that $\hat\varphi(\vec e)=w.$
Therefore, the  restriction of $\hat\varphi$ to the set of outgoing edges from  $p_i$ is injective.

On the other hand, the existence of two incoming edges $\vec {e_1}=(p_1,p)$,  $\vec {e_2}=(p_2,p)$ such that $\hat\varphi(\vec {e_1})=\hat\varphi(\vec {e_2})=w$ contradicts the injectivity of $\varphi|_{E(w)}$ in $\Gamma$. 
Therefore, the  restriction of $\hat\varphi$ to the set of edges of ${{\mathcal E}}$,  incoming  to  $p_i$, is injective.\\

\subsection{Walks, addresses, and the index diagram}\label{waid}

Recall from Section~\ref{prel} that any infinite sequence $\mathbf{\alpha}=\alpha_1\alpha_2\ldots\in I^\infty$ defines a point $x=\pi(\mathbf{\alpha})\in K$, and $\mathbf{\alpha}$ is called an \emph{address} of $x$. The goal of this section is to encode addresses of boundary points $p\in P$ as infinite walks in the index diagram $\mathcal{G}$.

\begin{definition}\label{walks}
Let $\mathcal{G}=(P,\mathcal{E})$ be a directed pseudograph.  A \emph{walk from $p_0$ to $p_k$} in $\mathcal{G}$ is a sequence
\[
    \omega(p_0,p_k) = (p_0,\vec{e_1},p_1,\ldots,\vec{e_k},p_k)
\]
where  $k\ge 0$, such that for each $i=1,\ldots,k$, there is a directed edge  $\vec{e_i}=(p_{i-1},p_i)$. 

 A walk is called a  {\em path}  if its vertices $p_0,\ldots,p_k$ are pairwise distinct. 
 In particular, the walk of length zero at $p$ is the {\em trivial path} at $p$.
  A walk of positive length is {\em closed} if $p_0=p_k$. 
  A closed walk is a {\em directed cycle} if $p_0,\ldots,p_{k-1}$ are pairwise distinct. 
  In particular, a loop is a directed cycle of length one. 
  A vertex is {\em cyclic}  if it lies on a directed cycle.

The set of all finite walks from $p$ to $q$ in $\mathcal{G}$ is denoted by $\Omega(p,q)$.\\
The set of all infinite walks $\omega(p,\ldots)$ starting at $p$   is denoted by $\Omega(p)$.\\
For any subsets $A,B$ of the set $P$, we put $\Omega(A,B)=\bigcup\limits_{a\in A, b\in B}\Omega(a,b)$. \\
$A$ and $B$ are {\em independent} in $\mathcal{G}$, if $\Omega(A,B)=\Omega(B,A)=\varnothing$. \\
 If $\Omega(A,B)\neq\varnothing$, then we write $A\hookrightarrow B$ and say that $B$ is \emph{reachable from} $A$.
 The set $\{x\in P: A\hookrightarrow x
 \}$ is called the reachability domain of  $A$; it always contains $A$.
\end{definition}

We may specify a non-trivial walk $\omega$ by its edge sequence $(\vec{e_1},\ldots,\vec{e_k})$, which uniquely determines its vertex sequence $(p_0,p_1,\ldots,p_k)$. We use the vertex sequence alone only when it unambiguously determines the edges of the walk. A trivial walk is specified by its single vertex $p_0$.

\medskip

If $\mathcal{G}$ is the index diagram of a sprout $(\Gamma, \varphi)$, we extend the edge labeling map $\hat{\varphi}\colon \mathcal{E}\to W$ to walks in $\mathcal{G}$  in the natural way:\\ for a finite walk $\omega=(p_0,\vec{e_1},p_1,\ldots,\vec{e_k},p_k)$, we define
\[
   \hat{\varphi}(\omega) = \hat{\varphi}(\vec{e_1})\hat{\varphi}(\vec{e_2})\ldots\hat{\varphi}(\vec{e_k}),
\]
and for an infinite walk $\omega=(p,\vec{e_1},p_1,\vec{e_2},p_2,\ldots)$, we define
\[
   \hat{\varphi}(\omega) = \hat{\varphi}(\vec{e_1})\hat{\varphi}(\vec{e_2})\ldots \in I^\infty.
\]
This extension will be used throughout the rest of the paper without introducing additional notation.

\subsection{The sprout defined by an IFS} 

\begin{definition}\label{difs}
A \textrm{D-IFS} is a system 
$\mathcal{S} = \{S_1, \ldots, S_m\}$ of injective contractions on a complete metric space $X$
whose attractor $K(\mathcal{S})$ is a self-similar dendrite 
with a finite self-similar boundary $\partial K$.
\end{definition}

\noindent By Proposition \ref{dsip}, the attractor of every D-IFS 
has the single intersection property (SIP).

Therefore, a D-IFS $\mathcal{S} = \{S_1,\ldots,S_m\}$ defines the $P$-sprout   $\Gamma(\mathcal{S})$ as the edge-labeled bipartite
graph $\Gamma=(W,B,E,\varphi)$ with
$W=\{K_1,\dots,K_m\}$, $B=C\cup P$, where $C$ is the critical set of
$\mathcal{S}$, $(K_i,b)\in E$ iff $b\in K_i$, and for $e=(K_i,p)$,
$p\in P$, the label $\varphi(e)$ is  unique $p'\in P$ with $S_i(p')=p$.

The index diagram $\mathcal{G}_P=(P,\mathcal{E},\hat\varphi)$ is constructed
from $\Gamma$ as in Definition~\ref{indiag}.

\begin{theorem}\label{wtoa}
Let $p \in P$ and let $\omega = (p, \vec{e}_1, p_1, \vec{e}_2, p_2, \ldots) \in \Omega(p)$ be an infinite walk in the index diagram $\mathcal{G}_P$ of a D-IFS $\mathcal{S} = \{S_1, \ldots, S_m\}$. Then:

\begin{enumerate}
    \item For any $n \geq 1$, $p = S_{w_1} \circ S_{w_2} \circ \cdots \circ S_{w_n}(p_n)$, where $p_n$ is the $n$-th vertex of the walk $\omega$.
    \item The sequence $\hat{\varphi}(\omega) = w_1w_2\ldots$ is an address of $p$, i.e., $\pi(\hat{\varphi}(\omega)) = p$.
\end{enumerate}
\end{theorem}

\begin{proof}
By construction of the index diagram, a directed edge $\vec{e} = (p_i, p_j)$ with label $\hat{\varphi}(\vec{e}) = w_k$ signifies that $S_k(p_j) = p_i$. Therefore,  for any $n \geq 1$ we have $p = S_{w_1} \circ S_{w_2} \circ \cdots \circ S_{w_n}(p_n)$, which implies that $p \in K_{w_1\dots w_n}$ for all $n$ and the sets $K_{w_1\ldots w_n}$ form a nested sequence with $\bigcap\limits_{n=1}^\infty K_{w_1\ldots w_n} = \{p\}$.
\end{proof}

Each address $\alpha=i_1i_2\dots$ defines a unique point $x=\pi(\alpha)\in K$,
while, by Theorem~\ref{wtoa}, every infinite walk in
$\mathcal{G}_P$ starting at $p\in P$ produces an address of $p$. Since two
distinct boundary points cannot share an address, this imposes the following
restriction on the sprouts under consideration.

\begin{definition}\label{def:admissible}
A $P$-sprout is \emph{admissible} if no two different boundary points have
the same address, i.e.\ the restriction of $\pi$ to
$\bigcup_{p\in P}\hat\varphi(\Omega(p))$ is injective.
\end{definition}

\begin{lemma}\label{lem:order-addresses}
Let  $K(\mathcal{S})$ be the attractor of a D-IFS
$\mathcal{S}=\{S_1,\dots,S_m\}$. \\
For any $x\in K$ and any two  distinct addresses $\alpha=i_1i_2\dots$, $\beta=j_1j_2\dots  $ of $x$  there are multiindices  $\bi=i_1i_2\dots i_n$ and $\bj=j_1j_2\dots j_n$ such that $K_\bi\cap K_\bj=\{x\}$.\\
For any $x\in K$, if $\#\pi^{-1}(x)$ is finite, then
$\operatorname{Ord}(x,K)\ge\#\pi^{-1}(x)$; if $\#\pi^{-1}(x)$ is infinite,
then $\operatorname{Ord}(x,K)=\infty$.
\end{lemma}

\begin{figure}[htp]
\begin{center}
\includegraphics[width=.7\textwidth]{ 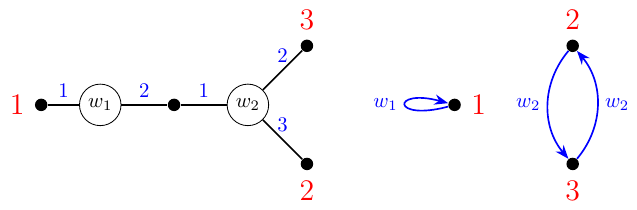}
\caption{Example of an inadmissible $P$-sprout (on the left) and its index diagram (on the right).}
\label{fig:unf}
\end{center}
\end{figure}

\section{Isomorphism theorem  for the attractors.}

This section is self-contained and independent of Sections 5–6; its main result is Theorem \ref{isom}

\begin{definition}\label{fipss}\cite{TYK}
Let $\mathcal K=\{K_i, i\in I=\{1, \ldots, m\}\}$ be a finite system of continua in a complete metric space $(X,d)$.  We say $\mathcal K$ has the {\em single intersection  property (SIP)},  if for any  $i\neq j\in I $, the set $P_{ij}=(K_i\cap K_j)$ is empty or a singleton.  
\end{definition}

In the setting of Definition \ref{fipss} we denote $|\mathcal K|=\bigcup\limits_{i\in I}K_i$,   $P=\bigcup\limits_{i\neq j}P_{ij} $, and $P_i=\bigcup\limits_{j\in I\setminus\{i\}}P_{ij}$.  Regarding $|\mathcal K|$ as the subspace of $X$ supplied by the induced topology, we see that for any $i$ the set $P_i$ is the boundary $\partial K_i$ of the set $K_i$ in $|\mathcal K|$,  and that its interior is $\dot K_i=K_i\setminus P_i$.  Observe that for any $i\in I$,  $\#\partial K_i\le m-1$.
If a point $x$ lies in some $P_i$, then we call $x$ the {\em boundary point} for $K_i$.\\

\begin{definition}\label{refin}\cite{TYK}
Let $\mathcal{K}=\{K_i,\ i\in I=\{1,\ldots,m\}\}$ and
$\mathcal{L}=\{L_j,\ j\in J=\{1,\ldots,n\}\}$ be SIP systems of continua in $X$.
We say that $\mathcal{L}$ is the \emph{refinement} of the system $\mathcal{K}$ if
\begin{enumerate}
  \item[(i)] for any $L\in\mathcal{L}$ there is $K\in\mathcal{K}$ such that $L\subset K$;
  \item[(ii)] for any $K\in\mathcal{K}$ the union $|\mathcal{L}_K|$ of the sets of a subsystem
        $\mathcal{L}_K=\{L\in\mathcal{L}: L\subset K\}$ is connected and contains $\partial K$.
\end{enumerate}\end{definition}

\begin{definition}\label{refsys}\cite{TYK}
Let $\{\mathcal{K}^n\}=\mathcal{K}^1,\mathcal{K}^2,\ldots$ be a sequence of
SIP systems of continua. If for any $n$ the system $\mathcal{K}^{n+1}$ is a refinement of
the system $\mathcal{K}^n$ and for any nested sequence $\{K^n_{i_n}\}_{n=1}^{\infty}$ such that
$K^n_{i_n}\in\mathcal{K}^n$ and $K^n_{i_n}\supset K^{n+1}_{i_{n+1}}$ the set
$\bigcap\nolimits_{n=1}^{\infty} K^n_{i_n}=\{x\}$ is a singleton,
then we call $\{\mathcal{K}^n\}$ a \emph{refinement sequence} for the set
$K=\bigcap\nolimits_{n=1}^{\infty}|\mathcal{K}^n|$.

{In what follows, the notation $K^n_i$ (resp.\ $\tilde K^n_i$) always
means that $K^n_i\in\mathcal{K}^n$ (resp.\ $\tilde K^n_i\in\tilde{\mathcal K}^n$):
the upper index indicates the level, the lower index the member.}
\end{definition}

\begin{lemma} \label{nbase}
Let $\{\mathcal K^n\}$ be { a refinement sequence} for the set $K$. Given a point $x\in K$, define 
$V_n(x)= \bigcup\limits_{K_{i_k}^n \ni x} K_{i_k}^n $. For any point $x \in K$  the family $\{V_n(x), n\in \mathbb{N}\}$ is a neighborhood base for $x$.
\end{lemma}

\begin{proof} Suppose contrary. Then there exists a neighborhood $U(x)$ such that for any $n \in \mathbb{N}$ there exists $K_{i_n}^n \ni x$ with $K_{i_n}^n \setminus U \neq \varnothing$.

For each $K_{i_n}^n$ we have a finite chain of continua $
K_{i_n^1}^1 \supset K_{i_n^2}^2 \supset \cdots \supset K_{i_n^n}^n.
$

This generates an infinite family of nested chains $\{K_{i_n^k}^k\}_{k=1}^n$. Since $\mathcal{K}^1$ is finite, by the pigeonhole principle, there exists $K_1^1 = K_{i_j^1}^1$ appearing infinitely often among $K_{i_n^1}^1$.

Within $K_1^1$, the elements $K_{i_n^2}^2$ form an infinite family, so by finiteness of $\mathcal{K}^2$ there exists $K_2^2 \in \mathcal{K}^2$ appearing infinitely often.

Continuing this process for each $\mathcal{K}^n$, we obtain an infinite nested sequence:
$$
K_1^1 \supset K_2^2 \supset \cdots \supset K_n^n \supset \ldots
$$
The sets $K_n' = K_n^n \setminus U$ form a nested sequence of non-empty compacts. Since $\bigcap_{n=1}^{\infty} K_n^n = \{x\}$, $\bigcap_{n=1}^{\infty} K_n' = \varnothing$,
which is a contradiction.  \end{proof}

\begin{proposition}\label{sip_hom}
Let  $\{\mathcal  K^n\}$  and $\{\tilde{\mathcal {K}}^n\}$ be   refinement sequences of SIP systems of continua for the sets $K$ and $\tilde{K}$, such that 
\begin{itemize}
\item[(i)] for any $n$, there is a bijection  $ \phi:\mathcal{K}^n\to  \mathcal{\tilde K}^n$   such that $\phi( K_i^n) =\tilde K_i^n $;
\item[(ii)]
$K_i^n\cap K_j^n\neq\varnothing $ iff $ \tilde  K_i^n\cap \tilde K_j^n\neq\varnothing$;
\item[(iii)]   
$K_i^n\supset K_j^{n+1}$ iff $ \tilde  K_i^n\supset \tilde K_j^{n+1}$.
\end{itemize}
Then there exists a homeomorphism $f : K \to \tilde{K}$ such that for any $i$ and $n \in \mathbb{N} $ $f(K_{i}^n) = \tilde{K}_{i}^n$. 
\end{proposition}

\begin{proof}

For any $x \in K$ there is a nested sequence $K^n_{i_n}$ such that
$x = \bigcap_{n=1}^{\infty} K^n_{i_n}$.
Put $f(x) = \tilde{x} = \bigcap_{n=1}^{\infty} \tilde{K}^n_{i_n}$.
By (iii), the intersection is a singleton, so $\tilde{x}$ is correctly defined.
If $x = \bigcap_{n=1}^{\infty} K^n_{j_n}$ for some $K^n_{j_n} \neq K^n_{i_n}$,
then for any $n$, $K^n_{i_n} \cap K^n_{j_n} = \{x\}$ implies
$\tilde{K}^n_{i_n} \cap \tilde{K}^n_{j_n} = \{\tilde{x}\}$, so the function $f$ is correctly defined.

Let $x = \bigcap_{n=1}^{\infty} K^n_{i_n}$ and $y = \bigcap_{n=1}^{\infty} K^n_{j_n}$ in $K$.
Clearly, $x \neq y$ iff there is $n \in \mathbb{N}$ such that
$K^n_{i_n} \cap K^n_{j_n} = \varnothing$, which is equivalent to
$\tilde{K}^n_{i_n} \cap \tilde{K}^n_{j_n} = \varnothing$ and therefore to
$\tilde{x} \neq \tilde{y}$. This shows that the map $f$ is a bijection.

The relation $x \in K^n_i$ iff $\tilde{x} \in \tilde{K}^n_i$ implies that
$f(K^n_i) = \tilde{K}^n_i$ for every $K^n_i \in \mathcal{K}^n$.

By Lemma \ref{nbase}, for any $x\in K$ the family
$V_n(x)=\bigcup_{K^n_{i_k}\ni x}K^n_{i_k}$ is a neighborhood base {\em at} $x$.
Since $x\in K^n_{i_k}$ iff $\tilde x\in \tilde K^n_{i_k}$, we have
$f(V_n(x))=\bigcup_{\tilde K^n_{i_k}\ni \tilde x}\tilde K^n_{i_k}
 =\tilde V_n(\tilde x)$,
which is a neighborhood base at $\tilde x=f(x)$ by Lemma  \ref{nbase}.
Thus $f$ maps the neighborhood base at every point onto the neighborhood
base at its image, so $f$ is open; applying the same argument with the
roles of $\mathcal K^n$ and $\tilde{\mathcal K}^{\,n}$ interchanged shows
that $f^{-1}$ is open, hence $f$ is continuous.

\end{proof}
 
 \begin{definition}
Let $\Gamma(W, B, E, \varphi)$ be a $P$-sprout. Put $W' = W \times W$ and $E' = W \times E$, and define $\varphi'(w \times e) = \varphi(e)$. Put $B' = (W \times B \cup B)/{\sim}$, where $w \times b_1 \sim b_2$ if $b_2 \in P$, $(w, b_1) \in E$, and $\varphi(w, b_1) = b_2$.

The $P$-sprout $\Gamma(W', B', E', \varphi')$ is denoted by $\Gamma^2$ (see Fig.\ref{G^2}). The index diagram $\mathcal{G}_P^2$ for the sprout $\Gamma^2$ has the same set $P$ of vertices; the edges are all walks $e_i e_j$ of length $2$ in $\mathcal{G}_P$.
\end{definition}

\begin{lemma}\label{sqrf}
Let   $K({\eS})$ be the attractor of a D-IFS ${\eS}=\{S_i,i\in I\}$  and let $\Gamma$ be its $P$-sprout. \\ The system $\{K_{ij},i,j\in I\}$ has the SIP and is the refinement of the system $\{K_i, i\in I\}$. (see Def. \ref{refin})\\ 
 The sprout $\Gamma^2$ is the $P$-sprout of the system ${\eS}^2=\{S_iS_j,(i,j)\in I^2\}$ and $\Gamma^2$ is the labeled intersection graph for the system of copies $K_{ij}$ of $K$. 
\qed
\end{lemma}

\begin{figure}[htp]\label{G^2}
\begin{center}
\includegraphics[width=.8\textwidth]{ 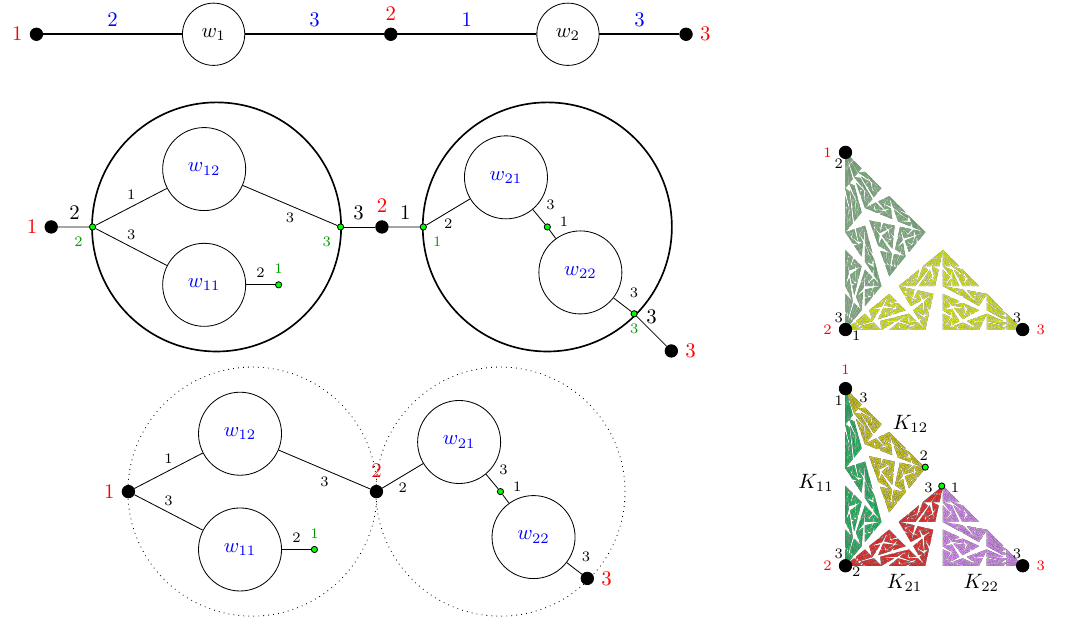}
\caption{\small The left column  shows how the sprout $\Gamma^2$(at the bottom) is obtained from the sprout $\Gamma$ (at the top). The  middle picture shows the neighborhoods of points 1,2 and 3 that shrink to each of these points. Right column shows the systems of  copies $K_i$ and $K_{ij}$. }
\end{center}
\end{figure}

\begin{corollary}\label{G2n}
If D-IFS ${\eS}=\{S_1,...,S_m\}$ and $\tilde {\eS}=\{\tilde S_1,...,\tilde S_m\}$ have isomorphic sprouts $\Gamma$ and $\tilde \Gamma$, then for any $n\in \mathbb{N}$, the sprouts $\Gamma^{2^n}$ and $\tilde{\Gamma}^{2^n}$ are isomorphic.
\qed
\end{corollary}

\begin{theorem}\label{isom}
   If the systems ${\eS}=\{S_1,...,S_m\}$ and $\tilde {\eS}=\{\tilde S_1,...,\tilde S_m\}$ have isomorphic $P$-sprouts $\Gamma$ and $\tilde \Gamma$, then their attractors $K$ and $ \tilde K$ are isomorphic.
\end{theorem}
\begin{proof}
 We may assume without loss of generality that
$\phi(K_i)=\tilde K_i$ for every $i\in I$.
For $n\in\mathbb N$ put
$\mathscr K^n=\{K_{\mathbf i}:\mathbf i\in I^{2^n}\}$ and
$\tilde{\mathscr K}^{\,n}=\{\tilde K_{\mathbf i}:\mathbf i\in I^{2^n}\}$.
We verify conditions (i)--(iii) of Proposition~\ref{sip_hom}
for the sequences $\{\mathscr K^n\}$, $\{\tilde{\mathscr K}^{\,n}\}$.

Since $K=\bigcup_{\mathbf i\in I^{2^n}}K_{\mathbf i}$, we have
$|\mathscr K^n|=K$; similarly $|\tilde{\mathscr K}^{\,n}|=\tilde K$, so
condition (i) holds. By Lemma~\ref{sqrf}, each of $\mathscr K^n$,
$\tilde{\mathscr K}^{\,n}$ is an SIP system of continua, and
$\mathscr K^{n+1}$ (resp.\ $\tilde{\mathscr K}^{\,n+1}$) is a refinement of
$\mathscr K^n$ (resp.\ $\tilde{\mathscr K}^{\,n}$). The multiindices of any nested chain
$K^1_{\mathbf i_1}\supset K^2_{\mathbf i_2}\supset\dots$ are successively
extended, hence determine a unique infinite sequence $\alpha\in I^\infty$ and 
$\bigcap_n K^n_{\mathbf i_n}=\{\pi(\alpha)\}$ is a singleton. Hence both
$\{\mathscr K^n\}$ and $\{\tilde{\mathscr K}^{\,n}\}$ are refinement
sequences.

By Lemma~\ref{sqrf}, $\Gamma^{2^n}$ is the labeled intersection graph of
$\mathscr K^n$: two distinct white vertices of $\Gamma^{2^n}$ are adjacent  to the same black vertex
iff the corresponding copies intersect. By Corollary~\ref{G2n},
$\Gamma^{2^n}$ and $\tilde\Gamma^{2^n}$ are isomorphic via the isomorphism
induced by $\phi$, which is the identity on $P$ and on the indices.
Therefore
$K_{\mathbf i}\cap K_{\mathbf j}\neq\varnothing$ iff
$\tilde K_{\mathbf i}\cap\tilde K_{\mathbf j}\neq\varnothing$,
which is condition (ii). Moreover,
$K^{n+1}_{\mathbf j}\subset K^n_{\mathbf i}$ iff $\mathbf i$ is a prefix of
$\mathbf j$; this relation depends only on the indices, hence it is
preserved identically, which is condition (iii).

By Proposition~\ref{sip_hom}, there is a homeomorphism
$f: K\to\tilde K$ with $f(K^n_{\mathbf i})=\tilde K^n_{\mathbf i}$ for
 all $n\in\mathbb N$ and $\mathbf i\in I^{2^n}$. 

Now, let $x\in K$ have an address $\alpha$. Then
$x=\bigcap_n K_{\alpha|_{2^n}}$, and since
$f(x)\in f(K_{\alpha|_{2^n}})=\tilde K_{\alpha|_{2^n}}$ for every $n$, we obtain $f(x)=\bigcap_n\tilde K_{\alpha|_{2^n}}$.
Hence, for $y=S_i(x)$,
$$
f(y)=\bigcap_n \tilde K_{i\cdot\alpha|_{2^n}}=\tilde S_i\Bigl(\bigcap_n \tilde K_{\alpha|_{2^n}}\Bigr)=\tilde S_i(f(x)),
$$
so $f(S_i(x))=\tilde S_i(f(x))$ for all $x\in K$ and $i\in I$, i.e.\ $f$ is
an isomorphism of the systems $\mathcal S$ and $\tilde{\mathcal S}$ in the
sense of Definition~\ref{isom1}.
\end{proof}

\section{The ramification points and the main tree of a self-similar dendrite.}
 
\subsection{The main tree and orders of cut points of $K$}
If $K$ is a dendrite, then {\em the order ${\mathop{\rm Ord} \nolimits}(x,K)$ of a point} $x\in K$ is equal to the number of connected components of $K\setminus\{x\}$.  \cite[Ch.6,Th. 6]{Kur}. 
     
The points of order 1 are called the {\em end points}, the points of  order $\ge$ 2 are  the {\em cut points}, and the points of order $\ge$ 3 are the {\em ramification points} of $K$.

\begin{definition}\label{mtree}\cite{STV}
Let  $K$ be the attractor of D-IFS ${\eS}=\{S_1,...,S_m\}$. 
The minimal subdendrite $\hga{\subset} K$ that contains  $\partial K$ is called the {\em main tree} of $K$.
\end{definition}
Since  $\partial K=\{p_1,...,p_n\}$ is finite and for any two points $x,y\in K$ there is a unique subarc $\gamma(x,y){\subset} K$,
the main tree $\hga$ is a finite union of all subarcs $\gamma(p_i,p_j){\subset} K$. Thus, $\hga$ is
a finite topological tree. As was proved in
\cite[Theorem 15] {STV}, the arcs $\gamma(p_i,p_j)$ are the components of the attractor of some multizipper, which is a graph-directed
system that has the SIP. The arcs $\gamma(p_i,p_j)$ are called {\em the main subarcs} of the dendrite $K$.

We consider the set $G_{\eS}(\hga)$ that is the union of all images $S_{\bf j}(\hga)$ of the main tree, and its subset
$G_{\eS}(C)=\bigcup\limits_{{\bf j}\in I^*} S_{\bf j}(C)$, that is, the set of critical points (\ref{dK})  and all their images.  It is essential because  
$G_{\eS}(C)=\{x\in K: \#\pi^{-1}(x)>1\}$, while $K\setminus G_{\eS}(C)$ is the set of points $x$ that have a unique address. The propeties of points with a unique address are stated in the following Proposition.

\begin{proposition}\label{ordx}
Let $K$ be the attractor of D-IFS ${\eS}$ and $\Gamma$ be its $P$-sprout.\\ 
If $x\in K\setminus G_{\eS}(C)$, then \\
1)  There is ${\bf {i}}\in {I^*}$ such that ${\mathop{\rm Ord} \nolimits}(x,K)={\mathop{\rm Ord} \nolimits}(x,S_{\bf {i}}(\hga))$;\\
2) ${\mathop{\rm Ord} \nolimits}(x,K)\le\#P$;\\
3) If $x$ is a ramification point, its address is preperiodic. 
\end{proposition}
\begin{proof} 1)  Let $Q_1,... Q_n$ be some finite set of connected components of $K \setminus \{x\}$. 
For each $k=1,...,n$ take  $x_k \in Q_k\setminus\{x\}$. There is ${\bf {i}}=i_1...i_n\in I^*$ such that
$x \in K_{\bf {i}}$ and
each $x_k\in K\setminus K_{\bf {i}}$. For each $k$ there is a unique arc $\gamma_k=\gamma(x,x_k){\subset} K$.
Note that $\gamma_k\setminus\{x\}{\subset} Q_k$ and
$\gamma_k\cap {\partial} K_{\bf {i}}=\{y_k\}$, where $y_k\neq x$. For different $k,l$,  
$\gamma(x,y_k)\cap \gamma(x,y_l)=\{x\}$, and all these arcs are  subarcs of
$S_{\bf {i}}(\hga)$. Therefore, 
${\mathop{\rm Ord} \nolimits}(x,K)={\mathop{\rm Ord} \nolimits}(x,K_{\bf {i}})={\mathop{\rm Ord} \nolimits}(x,S_{\bf {i}}(\hat\gamma))$ and $n\le\#P$, which gives 2).

3) Since  $x\notin  G_{{\eS}}(C)$, the point $x$ has a unique address. 
By 1), there is a copy $K_{\bf {j}}$ such that  ${\mathop{\rm Ord} \nolimits}(x,K)={\mathop{\rm Ord} \nolimits}(x,S_{\bf {j}}(\hga))$. 
If $x$ is a ramification point,  its predecessor $ S_{\bf {j}}^{-1}(x)$ 
is a ramification point of $\hga$ and lies in $\hga\setminus G_{{\eS}}(C)$.
Since the  number of  ramification points of $\hga$ is finite,  each ramification point $x\in K\setminus G_{{\eS}}({\partial} K)$ has a preperiodic address. 
\end{proof}

\begin{proposition}\label{cpgam}
Let $K$ be the attractor of D-IFS ${\eS}$ and $\Gamma$ be its $P$-sprout.\\
1) The set of cut points $CP(K)$ is contained in the set $ G_{\eS}(\hat \gamma)$;\\ 
2) $\hga{\subset} \bigcup\limits_{i=1}^m S_i(\hat \gamma$);\\
3) for any $x\in C$, ${\mathop{\rm Ord} \nolimits}(x,K)= \sum\limits_{j \in \tilde\varphi(x)}{\mathop{\rm Ord} \nolimits}(S_{j}^{-1}(x),K) $ 
 and
 
 $\pi^{-1}(x)=\bigcup\limits_{j: x\in K_j}j\cdot(\pi^{-1}(S_{j}^{-1}(x)))$.
 \end{proposition}

\begin{proof}
1) directly follows from \ref{ordx}, 1). 

2) Let $x\in\hga$ and $x\in S_i(K)$. If $x\notin {\partial} K_i$, there is a main subarc $\gamma(p_j,p_k)\ni x$. If any of these two points, say $p_j$, does not belong to ${\partial} K_i$, then $\gamma(x,p_j)\cap {\partial} K_i$ is a point  $\tilde p_j\in {\partial} K_i$. Thus, the arc $\gamma(\tilde p_j,\tilde p_k)=\gamma(p_j,p_k)\cap K_i$ is a  main subarc in  $K_i$. 

3) If $j\in \tilde\varphi(x)$, then $S_{j}^{-1}(x)$ is some boundary point $p$. In other words, Each of the addresses of the point $x$ is one of addresses of the point $p$ with an added prefix $j$. Taking union over all $j\in \tilde\varphi(x)$, we get the result.

\end{proof}

To evaluate the order of the boundary points of $K$, we consider the sets of  addresses of these points.
For any $p\in {\partial} K$,  $$\#\pi^{-1}(p)\le{\mathop{\rm Ord} \nolimits}(p,K)\le \#\pi^{-1}(p)\cdot(n-1)$$  if the set $\pi^{-1}(p)$ is finite, and this order is infinite if $\pi^{-1}(p)$ is infinite.

\subsection{The cardinality of $\Omega(p)$}

By Theorem \ref{wtoa}, the label sequence of every infinite walk starting at $p\in P$ is an address of $p$. Conversely, the successive predecessors of $p$ along any address remain in $P$ by the definition of the self-similar boundary and thus determine an infinite walk in $G_P$. Two distinct walks starting at $p$ have different labels at the first edge where they differ. Hence $\hat\varphi$ is a bijection from $\Omega(P)$ onto $\pi^{-1}(p)$.

Let $\mathcal C$ be the set of strongly connected components of $G_P$. The {\em condensation multigraph} of $G_P$  is the directed multigraph with vertex set $\mathcal C$ and one edge from $C$ to $D$ for each edge of $G_P$ whose tail lies in $C$ and whose head lies in $D$, where $C\ne D$. Thus parallel edges between components are retained, while edges within a component are omitted. 

Since the condensation multigraph is finite and acyclic, every infinite walk in $G_P$  eventually stays in a single component.

We partition the set $\mathcal{C}$ into three classes:
\begin{enumerate}
    \item   the set $\mathcal{C}_0$ of single-vertex trivial components without loops;
    \item the set $\mathcal{C}_1$ of components whose internal edges form a directed cycle;
    \item the set $\mathcal{C}_2=\mathcal{C}\setminus(\mathcal{C}_0\cup\mathcal{C}_1)$.
\end{enumerate}

For $p\in P$ and $C\in\mathcal C_1$, let $\Theta(p,C)$ denote the set of directed paths starting at $p$ whose final vertex is their only vertex in $C$. In particular, if $p\in C$, then $\Theta(p,C)$ consists only of the trivial path at $p$.

For every $p\in P$, let $\mathcal{C}_1(p)=\{C\in \mathcal{C}_1:  p\hookrightarrow C\}$. 

\begin{theorem}\label{inf_ram}
The set $\partial K$ is a disjoint union $P_c\sqcup P_w \sqcup P_f$, where
\begin{enumerate}
    \item    $p\in P_c$  if and only if there is  $C\in \mathcal{C}_2$ such that $p\hookrightarrow C$; in this case, the set $\Omega(p)$ is uncountable; 
    \item   $p\in P_w$ if and only  if  $p\notin P_c$ and there are distinct $C_1,C_2\in \mathcal{C}_1$ such that $p\hookrightarrow C_1$ and $C_1\hookrightarrow C_2$; in this case, the set $\Omega(p)$ is countably infinite ; 
    \item   For every $p\in P_f$, the set $\Omega(p)$ is finite and \[\#\Omega(p)=\sum\limits_{C\in\mathcal{C}_1(p)}\#\Theta(p,C).\]
    \end{enumerate}
   
\end{theorem}

\begin{proof}
(1) If  a strongly connected  directed pseudograph  $C$ is not a directed cycle, then the set of walks in the pseudograph $C$ is uncountable. Indeed,  there is a vertex $q\in V(C)$ with at least two distinct outgoing edges that belong to $C$. Therefore, there exist two closed walks $\sigma_1$ and $\sigma_2$ based at $q$ that start with these edges. Each infinite sequence  $i_1i_2...$ of digits 1 and 2  defines an infinite  walk $\sigma_{i_1}\sigma_{i_2}...$; the set of such sequences is uncountable. Therefore, the set of walks in $C$ is also uncountable.

Let $C\in\mathcal C_2$ be reachable from $p$. Every vertex of $C$ has at least one outgoing edge within $C$. If each vertex had exactly one such edge, strong connectivity would force $C$ to be a directed cycle. Thus some vertex $q\in C$ has two distinct outgoing edges within $C$. By strong connectivity, each of these edges can be extended to a closed walk based at $q$, giving two closed walks $\sigma_1$  and $\sigma_2$ with distinct first edges. Distinct infinite sequences of choices of $\sigma_1$ and $\sigma_2$ give distinct infinite walks: at the first differing choice, the walks take different edges. Prepending a fixed walk from $p$ to $q$ gives uncountably many walks in $\Omega(p)$.

(2)  If $p\notin P_c$, the component in which an infinite walk from $p$ eventually stays must belong to $\mathcal{C}_1$. Once a finite prefix ending in that component is fixed, there is exactly one continuation that stays in the component. Since there are only countably many finite prefixes, $\Omega(p)$ is at most countable. If $p\in P_w$, fix a walk from $p$ to the tail of an exit edge of $C_1$ and a walk from that edge to $C_2$. Taking any number of full turns around $C_1$ before taking the exit edge, and then staying in $C_2$ after reaching it, gives infinitely many distinct walks. Hence $\Omega(p)$ is countably infinite.

(3) Let $p\in P_f$. Every infinite walk from $p$ eventually stays in some $C\in\mathcal C_1(p)$. Its prefix up to the first entry into $C$ cannot repeat a vertex: otherwise it would visit a distinct reachable cyclic component from which $C$ is reachable. Thus this prefix belongs to $\Theta(p,C)$. Conversely, each path in $\Theta(p,C)$ extends to an infinite walk by following the cycle in $C$. This is its only infinite extension, since an infinite walk leaving $C$ would eventually stay in a distinct reachable cyclic component. Hence the paths in $\Theta(p,C)$ are in bijection with the infinite walks from $p$ that eventually stay in $C$. Summing over $C\in\mathcal C_1(p)$ gives the stated formula. The sum is finite because $G_P$ has finitely many vertices and edges.
 \end{proof}

\begin{figure}[htp]
\centering
\includegraphics[width=.9\linewidth]{ 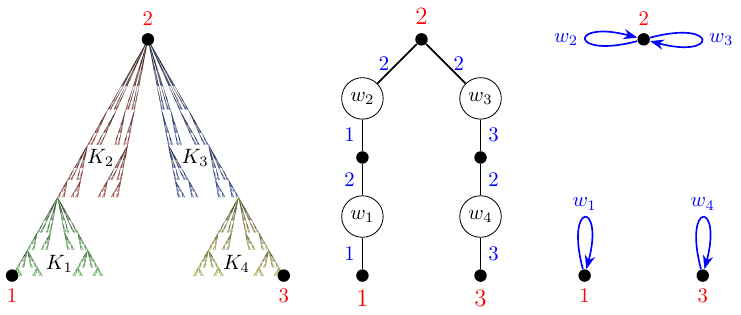}
\caption{A self-affine dendrite $K$ (left), its $P$-sprout (center) and the index diagram (right). The set of addresses of $p_2 \in \partial K$ is uncountable by  Theorem \ref{inf_ram}/1. Each of the points $p_1, p_3 \in \partial K$ has a  unique address by   Theorem \ref{inf_ram}/3. }
\label{fig:infram_index}
\end{figure}

\section{Finding the order of a point $x\in\hga$ with respect to the main tree.}

In the following, we  assume that $\Gamma=\Gamma(W,B,E,\varphi)$ is a {\em regular} $P$-sprout   with the sets of white vertices $W=\{w_1,..,w_m\}$, of indices $I=\{1,..,m\}$,  of black vertices $b\in B$ and of boundary points $p\in P$.\\

For each subset $P'$ of the set $P$, there is the smallest subtree  $\gamma(P') \subset \hga$ that contains the subset $P'$. 
The set of all the  subtrees  $\gamma(P')\subset \hga$ will be denoted by $\mathcal{P}(\gamma)$ or simply by $\mathcal{P}$.\\
The smallest subtree of the graph $\Gamma=\Gamma(P)$ that contains the subset $P'$ will be denoted by $\Gamma(P')$.
A subset $P'\IN P$ is {\em full} if $\gamma(P') \cap P=P'$, which is equivalent to $\Gamma(P') \cap P=P'$. By default, we consider only full subsets $P'\IN P$.

\subsection{ Definition of the mapping $\phi_i$} Given $w_i\in  W$, we define the mapping $\phi_i:P\to P$ as follows.
For any   $p\in P$  there is a unique path $\beta(p,w_i)$  in $\Gamma$ with endpoints $p$ and $w_i$.
This path is an alternating sequence of  black and white vertices  $b_{i_0} w_{j_0}....b_{i_k} w_{j_k} $ 
 where $b_{i_0}=p$ and $w_{j_k}=w_i$,  such that each two consecutive elements are the endpoints 
 of some edge in $\Gamma$. Let the last edge of this path be $e=(b_{i_k} w_{j_k})$   and  let $\varphi(e)$ be the label on $e$. Then  put $\phi_i(p):=\varphi(e)$.\\
 
 There may be a degenerate case, where some $p_k\in P$ is incident to $w_i$, then $p_k\in K_i$ and $\phi_i(p_k)=S_i^{-1}(p_k)$.\\

We use the notation $\#_c(A)$ for the cardinality of the set of connected components of a set  $A$.\\

\begin{lemma}\label{degwi}
For any $i\in I$,  $\#\phi_i(P)= \deg(w_i, \Gamma)$ and
$$\#_c(K\setminus K_i)=\#_c(\hga\setminus K_i)=\#\phi_i(P) -\#(K_i\cap \partial K). 
$$
\end{lemma}
\begin{proof}
By its definition, $\phi_i(P)$ is the set of labels on the edges, incident to $w_i$. Since $\Gamma$ is regular, 
$\Gamma\setminus \{w_i\}$ is a disjoint union of components $C_k$, each containing one of the edges $e_k$ and a nonempty subset $P_{e_k}\IN P$ such that $P=\bigsqcup P_{e_k}$.

 In the  degenerate
case, a component $C_k$ contains only an edge $e_k=(p_k,w_i)$ and the boundary point $p_k$, and $P_{e_k}=\{p_k\}$. In that case $p_k\notin K\setminus K_i$ and $p_k\notin \hga\setminus K_i$.

If $C_k$ is non-degenerate $p_k\notin P_{e_k}$, there is a component of $\hga\setminus K_i$ and a component of   $K\setminus K_i$ that contain $P_{e_k}$ and $S_i(\phi_i(p))$.
    
\end{proof}

\begin{figure}[htp]\label{Fy1} 
\centering
\includegraphics[width=.80\textwidth]{ 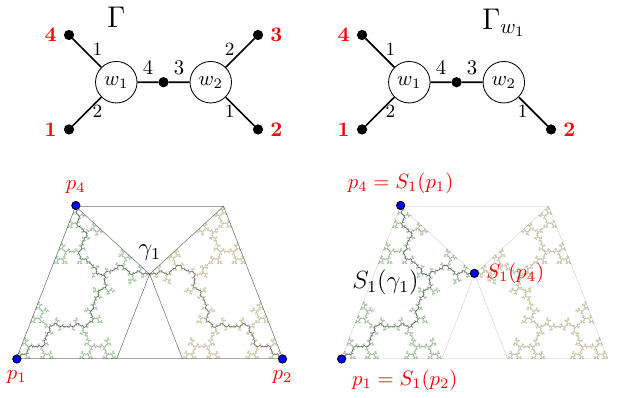}
\caption{}
\end{figure}

For the transformation $\phi_i:P\to P$,  denote the set $\phi_i(P)$ by $P_i$,  the subgraph $\Gamma(P_i)$ by $\Gamma_i$ and  the subtree $\gamma(P_i)$ by $\hga_i$.\\

Notice that $\Gamma_i$ is the minimal subgraph of $\Gamma$ that contains $P_i$ and $\hga_i$ is a minimal topological subtree of $K$ that contains $P_i$.\\

{\small For the sprout $\Gamma$ in Figure \ref{Fy1},  $\phi_1 = \begin{pmatrix}
1 & 2 & 3 & 4  \\
2 & 4 & 4 & 1
\end{pmatrix} $ and $\phi_1(P) = \{p_1,p_2,p_4\}$. We show $\Gamma_1$ in the top-right, $\hga_1$ in  the bottom-left, and  
$S_1(\hga_1) = K_1 \cap \hat{\gamma}$ in  the bottom-right of the Figure.}

The definition of   map $\phi_i$ has an equivalent formulation in terms
of subarcs of the main tree $\hga$, which easily extends to the maps $\phi_{\bf j}$ for any ${\bf j}\in I^*$.

\begin{lemma}
Let $\gamma(p, K_i)$  be the
 minimal arc connecting $p$ and $S_i(K)$ in  $K$.  
 Then its endpoint in $K_i$,  i.e. $\gamma(p, K_i)\cap\partial K_i$, is $S_i(\phi_i(p))$. \hfill $\square$
\end{lemma}

 In other words, $\phi_i(p)=S_i^{-1}(\gamma(p, K_i)\cap\partial K_i)$. 
 
 Also, for any full $P'\subset P$, 
$S_i(\phi_i(P')) \subset\gamma(P')\cap \partial K_i$. In particular,

 $S_i(P_i)=\hga\cap \partial K_i=S_i(\hga_i) \cap \partial K_i$.\\

 \begin{proposition}\label{fybj}
 Let ${\bf j}=j_1...j_k\in I^k$ and put $\phi_{\bf {j}}=\phi_{j_k}\cdot...\cdot\phi_{j_1}$.\\
Let $\gamma(p, K_{\bf j})$  be the
 minimal arc connecting $p$ and $S_{\bf j}(K)$ in  $K$.  \\
 Then its endpoint in $K_{\bf j}$  is $\gamma(p, K_{\bf j})\cap\partial K_{\bf j}=S_{\bf j}(\phi_{\bf j}(p))$.
 
 \end{proposition}
\begin{proof}
Let $i,j\in \{1,....,m\}$. Consider the minimal arc $\gamma(p,K_{ij})$. 
This arc is a union of minimal arcs $\gamma(p, K_i)$ and $S_i(\gamma(\phi_i(p), K_{j})) $.  
These arcs have a common endpoint $S_i(\phi_i(p))$. 
Since 
$\gamma(\phi_i(p), K_{j})\cap K_j=\{S_j(\phi_j\cdot \phi_i(p))\}$, 
the endpoints of $\gamma(p,K_{ij})$  are $p$ and $S_{ij}(\phi_j\cdot\phi_i(p))$. 
Therefore,  $\phi_{ij}=\phi_j\cdot\phi_i$.\\
By induction, we ensure that for each ${\bf {j}}=j_1...j_k$,  $\phi_{\bf {j}}=\phi_{j_k}\cdot...\cdot\phi_{j_1}$. 
For a minimal arc $\gamma(p,K_{\bf j})$ 
we obtain $\gamma(p,K_{\bf j})\cap K_{\bf j}=S_{\bf j}(\phi_{\bf j}(p))$. \end{proof}

 {\bf Notation: $P_{\bf j},\Gamma_{\bf j}$ and $\hga_{\bf j}$.} For any full $P'\subset P$ and ${\bf{j}}\in I^*$,\\
$S_\bj(\phi_\bj(P'))=\gamma(P')\cap \partial K_\bj$, therefore $S_{\bf{j}}(P_{\bf{j}})=\hga \cap \partial K_{\bf{j}}$.\\
Hence, we denote the set $\phi_{\bf{j}}(P)$ by $P_{\bf{j}}$,  the subgraph $\Gamma(P_{\bf{j}})$ by $\Gamma_{\bf{j}}$, and  the subtree $ \gamma(P_{\bf{j}})$ by $\hga_{\bf{j}}$.\\
As  follows from Proposition \ref{fybj},
$P_{\bf{ij}}=\phi_{\bf{ij}}(P)=\phi_{\bf{j}}(P_{\bf{i}})$.\\

 {\bf The mapping $\phi^*_i$}.  The 
 equality $\gamma(p,K_{\bf j})\cap K_{\bf j}=S_{\bf j}(\phi_{\bf j}(p))$ leads us to introduce a  mapping
$\phi^*_i:\mathcal{P}(\gamma)\to
\mathcal{P}(\gamma)$ defined by the formula\\
$\phi^*_i(\gamma(P'))=\gamma(\phi_i(P'))$ for each  full $P'\subset P$, such that $w_i \in \Gamma(P')$ .\\
Rewriting its right-hand side   gives
$\phi^*_i( \gamma(P'))=S_i^{-1}( \gamma(P')\cap K_i)$ for each $P'$ such that $w_i \in \Gamma(P')$. Proceeding  to multiindices, we obtain that for any   $\bf {j}$,
$S_{\bf {j}}(\phi_{\bf {j}}^*(\hga))=\hga\cap K_{\bf {j}}$, or, in short,   $\phi^*_{\bf {j}}(\hga)=\hat \gamma_{\bf {j}}$.

  If $x, y \in \hat{\gamma}$ and $y =S_{\bf{j}}(x), \bf{j} \in $ $I^*$, then 
   \begin{equation}\label{eq:ord}
{\mathop{\rm Ord} \nolimits}(y,\hga) \geq {\mathop{\rm Ord} \nolimits}(y,K_{\bf{j}} \cap \hga) =  {\mathop{\rm Ord} \nolimits}(x,\phi^*_{\bf{j}}(\hga)).
\end{equation}
Therefore, for any ${\bf i,j}\in I^*$,
      $\#\phi_{\bf{i}}(P) \geq \#\phi_{\bf{ij}}(P) $.

\begin{proposition}\label{ordim}
  For $b\in B$,  $j\in I$, ${\bf{i}}\in I^*$, the following properties are fulfilled 
  
  1) If $\#\phi_{{\bf{i}}j}
    (P)=1$, then $w_{j} \notin \Gamma_{\bf{i}}$,

  2) If $\#\phi_{{\bf{i}}j}
    (P)=k>1$, then $\deg(w_{j},\Gamma_{\bf{i}}) = k$,
    
    3)  If $\deg(b, \Gamma) = k>1$, then $\Ord(b,\hga) \geq k$, 
    
4) If $\deg(b, \Gamma_{\bf{i}}) =k $, then $\Ord(y,\hga) \ge  k$, where $y = S_{\bf{i}}(b)$.
\end{proposition}
\begin{proof}
  
1) If $\#\phi_{{\bf{i}}j}(P) = 1$, then $\deg(w_j, \Gamma_{\bf{i}}) = 1$ therefore $w_{j} \notin\Gamma_{\bf{i}}$.   

 2) Since $\phi_{{\bf{i}}j}(P) = \phi_{j}( \phi_{\bf{i}}(P))$,   $\deg(w_{j}, \Gamma_{\bf{i}}) = \#\phi_{j}(\phi_{\bf{i}}(P))$.

3) If $\deg(b, \Gamma) = k$  there exist  subarcs  $\gamma(p_1,b), \gamma(p_2,b),\dots \gamma(p_k,b)$ of $\hat{\gamma}$ such that for any $i\neq j$, \  $\gamma(p_i,b)\cap \gamma(p_j,b) = \{b\}$. It follows that $\Ord(b,\hga) \geq k$.

4)  If $\deg(b, \Gamma_{\bf{i}}) \geq k > 1$ there are subarcs $\gamma(p_i,b), i=1,...,k$ of $\phi^*_{\bf{i}}(\hat{\gamma})$ with the same property as in 3). It follows that $\Ord(b,\phi^*_{\bf{i}}(\hat{\gamma})) \ge k$. 
By formula \eqref{eq:ord}, 
$\Ord(S_{\bf{i}}(b),\hat{\gamma}) \ge\Ord(b,\phi^*_{\bf{i}}(\hat{\gamma})) \geq k$. Therefore, $\Ord(y,\hat{\gamma}) \ge  k$.
By formula \eqref{eq:ord}, 
$\Ord(S_{\bf{i}}(b),\hat{\gamma}) \ge  \Ord(b,\phi^*_{\bf{i}}(\hat{\gamma})) \geq k$. Therefore, ${\Ord} (y,\hat{\gamma}) \ge  k$.
\end{proof}

{\bf Definition and properties of     $\Nfi(\al)$. }\ For an address $\al=j_1\ldots j_k\ldots $, we define $\al|_n=j_1\ldots j_n $. Consider the sets $P_{\al|_n}=\phi_{\al|_n}(P)$.
For any $n$, $\#P_{\alpha|_n}\ge \#P_{\alpha|_{n+1}}\ge 1$, and hence
$\lim\limits_{n\to\infty}\#P_{\alpha|_n}=\min \#P_{\alpha|_n}$ is an integer $\ge 1$, which we denote by $\Nfi(\al)$.

\begin{theorem}\label{ord_b}

For any $x\in K$, the set $A_x=\{\alpha\in\pi^{-1}(x):\Nfi(\alpha)  > 1\} $ is finite.  $A_x=\varnothing$ if and only if $x\notin\hat\gamma$.
\begin{itemize}
   
     \item [1)] If $\#A_x=1$ and $x\notin  \partial K$, then  ${\mathop{\rm Ord} \nolimits}(x,\hga) =  \Nfi(\alpha) $.
      \item [2)] If $\#A_x=1$ and $x\in  \partial K$, then  
     ${\mathop{\rm Ord} \nolimits}(x,\hga) =  \Nfi(\alpha)  - 1$.
     \item [3)] 
If $\#A_x>1$, then 
    \begin{equation}
      {\Ord}(x,\hga)= \sum\limits_{\alpha\in A_x} (\Nfi(\alpha)  - 1). 
    \end{equation}
     \end{itemize}
\end{theorem}
 
\begin{proof}
1)  If for some address $\alpha$ of $x\;$ $N_{\phi}(\alpha) = 1$ there is sufficient $n \in \mathbb{N}$ such that for multiindex $j_1j_2\dots j_n = \bf{j} \sqsubset \alpha\;$, $\#\phi_{\bf{j}}(P) = 1$. In that case, every minimal arc that connects a boundary point and a copy $K_{\bf{j}}$ will pass through the point $x$. Therefore, $K_{\bf{j}}\cap \hat{\gamma} = \{x\}$. It follows that $\Ord(x, K_{\bf{j}} \cap \hat{\gamma}) = 0$.

If the point $x \in \hga \setminus\partial K$ and $ A_x =\{\alpha\}=$ $j_1j_2\hdots$, then $x \in \hga \setminus G_{{\eS}}(\partial K)$  and $\Nfi(\alpha)=k > 1$. There is  $m$ such that for the initial subword $\bf{j} =$ $ j_1j_2\hdots j_m$ of $\alpha$,   $K_\bj\cap\partial K=\varnothing$ and $\#\phi_{\bf{j}}(P) = k$. 

The same holds for any $\bf{j}'$ such that $ \bf{j} \sqsubset \bf{j}' \sqsubset \alpha$. Therefore, $K_{\bf{j}'} \cap \hga$  is a subtree of $\hga$ that has $k$ endpoints and a unique  point $x$. Therefore,  ${\mathop{\rm Ord} \nolimits}(x,\hat{\gamma}) = \Nfi(\alpha) = k$.

 An example of this situation is shown in Fig. \ref{fig:case3ex}.
\begin{figure}[htp]
    \centering
    \includegraphics[width=.8\linewidth]{ 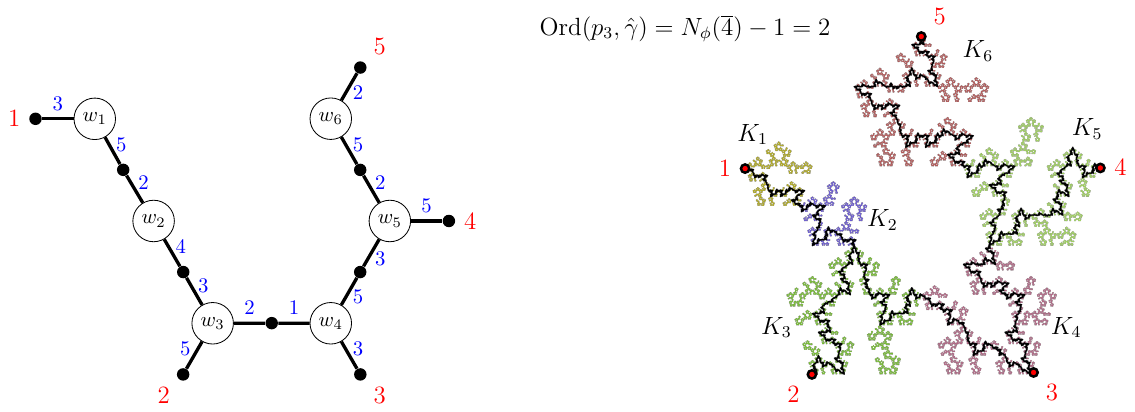}
    \caption{ Attractor $K$ and its $P$-sprout. The boundary point $p_3$ has one address. Its order in $\hat{\gamma}$ is equal to 2. }
     
    \label{fig:case3ex}
\end{figure}

2) In the case where $x \in \partial K$, $ A_x =\{\alpha\}=j_1j_2...$, and $\Nfi(\alpha)=k > 1$, there is  $m$ such that for the initial subword $\bf{j} =$ $ j_1j_2\hdots j_m$ of length $m$,   $K_{\bf{j}}\cap \partial K=\{x\}$ and $\#\phi_{\bf{j}}(P) = k$.

These two relations hold for any $\bf{j}'$ such that $ \bf{j} \sqsubset \bf{j}' \sqsubset \alpha$. Then $K_{\bf{j}'} \cap \hat \gamma $  is a subtree of $\hga$ that has $k-1$ endpoints and a unique cut  point $x$, because $x \in \phi_{\bf{j'}}(P) $. Therefore,  ${\mathop{\rm Ord} \nolimits}(x,\hat{\gamma}) = \Nfi(\alpha)-1 = k-1$.

3) Let $\alpha^1,...,\alpha^n$ be the addresses of the point  $x$, such that for any   $i \in \mathbb{N}$, $\Nfi(\alpha^i) > 1$.
For a sufficiently large $m$,  a set $\{\bf{j}^1,\bf{j}^2...,\bf{j}^n\}$ of initial subwords of  length  $m$ for the  addresses $\alpha^1,...,\alpha^n$ satisfies the following conditions.\\ a) All these multiindices are different.\\ b) For each $\bf{j}^k$, $K_{\bf{j}^k} \cap \partial K = \{x\}$ if $x \in \partial K$ and $K_{\bf{j}^k} \cap \partial K = \{\varnothing\}$ if $x \in \hat{\gamma}\setminus \partial K$.\\ c) Each copy $K_{\bf{j}^k}$  does not contain any boundary or ramification point of  $\hga$ except $x$. 

Under these conditions, $\#\phi_{\bf{j}^k}(P) = \Nfi(\alpha^k)$.
Then for any $\bf{j}'^k $ such that $\bf{j}^k \sqsubset \bf{j}'^k \sqsubset \alpha_k$,\ \ 
$\#S_{\bf{j}'^k} (P) \cap \hat \gamma = {\mathop{\rm Ord} \nolimits}(x, K_{\bf{j}'^k }  \cap \hat \gamma) + 1 $.

Applying the above argument to each of the addresses of the point   $x$, we obtain the inequality 
 $\sum\limits_{k=1}^n {\mathop{\rm Ord} \nolimits}(x, K_{{\bf{j}^k}} 
\cap \hat \gamma) \le  {\mathop{\rm Ord} \nolimits}(x, \hat \gamma) \le \#P   $. Therefore, the set $A_x$ is finite and  ${\mathop{\rm Ord} \nolimits}(x,\hga) = \sum\limits_{\alpha \in A_x} (\Nfi(\alpha)  - 1)$.\\
If for some address $\alpha^i$ there is a multiindex $\bf{i}$ $\sqsubset \alpha^i$ such that $\#\phi_{\bf{i}}(P) = 1$, then $K_{\bf{i}} \cap \hat \gamma = \{x\}$. Consequently, ${\mathop{\rm Ord} \nolimits}(x, K_{\bf{i}} \cap \hga) = 0 $.\\
\end{proof}

\begin{figure}[htp]
\centering
\includegraphics[width=.8\textwidth]{ 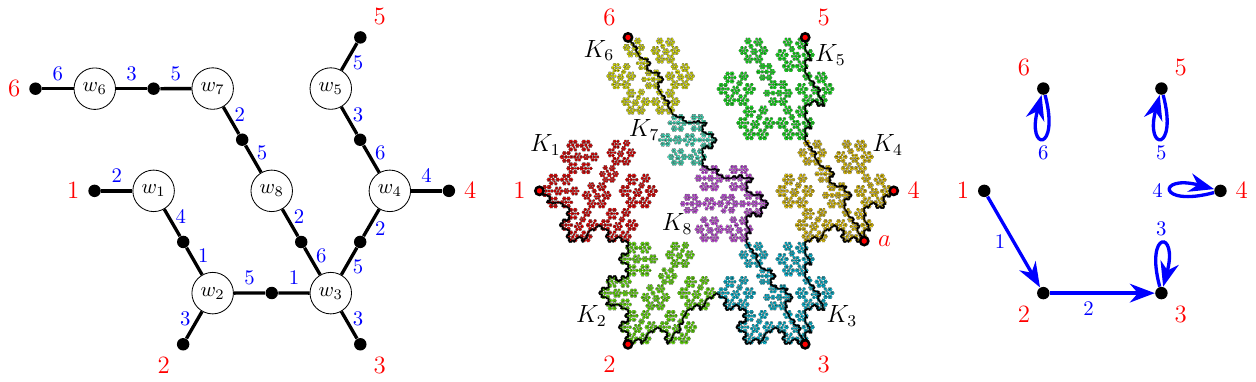}
\caption{ An attractor $K$ with its main tree $\hat \gamma$, its $P$-sprout, and its index diagram. Each boundary point has a  unique address. ${\mathop{\rm Ord} \nolimits}(p_2, \hat \gamma) = \Nfi(2\overline{3}) -1= 2, {\mathop{\rm Ord} \nolimits}(p_3, \hat \gamma) = \Nfi(\overline{3}) -1= 3$. All the other points are endpoints. There is ramification point $a$, such that $\pi^{-1}(a) = 4\overline{3},\; {\mathop{\rm Ord} \nolimits}(a, \hat \gamma) = \Nfi(4\overline{3}) = 3 $ }
\label{fig:hex_polygon2}
\end{figure}

 \begin{corollary}\label{ram_adr}
        Let  $\alpha$ be the address of a point $x\in K$ and $\Nfi(\alpha) \ge 3$.\\ If $\pi^{-1}(x)=\{\alpha\}$, $x\in\partial K$ and $\Nfi(\alpha) = 3$, then ${\rm Ord}(x,\hga)=2$.\\ Otherwise,  ${\mathop{\rm Ord} \nolimits}(x,\hga) \ge 3$.
\end{corollary}

\begin{proposition}\label{ram_p}
       If  for some $\bf{i} $, $\#\phi_{\bf{i}}(P) \ge 3$, then at least  one of the following statements holds:\\
1) $K_{\bf{i}}\cap \partial K$  is nonempty and contains a point  $x$ such that  ${\mathop{\rm Ord} \nolimits}(x,\hga) \geq 2$;\\
2) there is a point  $x \in K_{\bf{i}} \cap \hat{\gamma}$ such that  ${\mathop{\rm Ord} \nolimits}(x,\hga) \geq 3$.
\end{proposition}

\begin{proof}
If there is $x\in K_\bi\cap \partial K$, then there are $ \{a,b\} \IN (\partial K_\bi \cap \hga)\setminus\{x\}$. Consider the subarcs $\ga(a,x), \ga(b,x)$ of $\hga$. If $\ga(a,x)\cap \ga(b,x)=\{x\}$, ${\Ord}(x,  \hga) \ge 2$, and 1) follows. \\
Otherwise, there is $y$ such that $\gamma(a,x)\cap \gamma(b,x)=\gamma(y,x)$, and
${\mathop{\rm Ord} \nolimits}(y,  \hga) \geq 3$.\\
If $K_{\bf{i}}\cap \partial K=\varnothing$, take $ \{a,b,c\} \subset \partial K_{\bf{i}} \cap \hat{\gamma}$. 
 Consider the subarcs $\gamma(a,b)$, $\gamma(b,c)$, $\gamma(a,c)$ of $\hat{\gamma}$. The  intersection  $\gamma(a,b) \cap \gamma(b,c) \cap \gamma(a,c)$ is a singleton $\{x\}$.
 If $x \notin \{a,b,c\}$, ${\mathop{\rm Ord} \nolimits}(x, \gamma(a,b,c)) = 3$.\\
If $x \in \{a,b,c\}$, say $x=b$, then $\gamma(a,b) \cap \gamma(b,c) = \{b\}$. Since there is a subarc $\gamma'\subset\hga\setminus{K_{\bf{i}}}$  ending at  $b$, ${\mathop{\rm Ord} \nolimits}(x, \hga) \ge {\mathop{\rm Ord} \nolimits}(x,  (\gamma'\cup \gamma(a,b) \cup \gamma(b,c)) ) =3$.
\end{proof}
Proposition \ref{ram_p}  shows how to find  copies that contain the ramification points and thereby find the addresses of these points. In case the ramification points have multiple addresses, they are the images of boundary points. The addres\-ses of  boundary points are defined using the index diagram according to Pro\-po\-sition \ref{infwalks}. A tool for finding all the needed addresses will be a directed  graph $G_T$, which we define  below.\\

{\bf The graph $G_T$}. By analogy to the index diagram, we construct a labeled digraph $G_T=\{ V_T, E_T, \psi_T\}$,  such that each proper walk in this graph defines an address of a ramification point of $\hga$ or of a  boundary cut point in  $\hga$.\\

 The set of vertices  $V_T$ consists of three disjoint parts: $V_Q$,  
 $V_B$, and $V_P$. \\
 $V_Q$ is the set of all subsets
$Q\IN P$  such that  $\#Q\ge 3$ and for some $\bi\in I^*$, $\phi_\bi(P)=Q$; \\
$V_B=\{b\in B: \deg(b, \Gamma) \ge  3\}$;\\
$V_P=\{p\in P: \deg(p, \Gamma) \ge  2\}$. 
\\

The set of edges $E_T$ also consists of 3 disjoint parts, $E_Q$, $E_B$, and $E_P$.\\
The set $E_Q$ is supplied with a labeling function $\psi:E_Q\to I$.

A pair $Q,Q'\in V_Q$ defines  an edge $\be=(Q,Q')\in E_Q$ with a label $\psi(\be)=i$ exactly when there is $i\in I$ such that $\phi_i(Q)=Q'$. 

In the same way as in  the definition \ref{walks}, each walk $\be_1...\be_k$, $\be_i\in E_Q$,   has labels $l_j=\psi(\be_j),j=1,...,k$ and travels along a sequence  of vertices
 $Q_0,...,Q_k$ such that for any  $j=1,...,k$, 
 $\phi_{l_j}(Q_{j-1})=Q_j$. Consequently,  $\phi_{l_1....l_k}(Q_0)=Q_k$.
 
 In its turn, each infinite walk $\be_1...\be_k...$ in $E_Q$ that starts from $Q_0=P$ defines an address ${\beta }=l_1l_2...$ such that $\Nfi(\beta)\ge 3$.

 The sets of edges $E_B$ and $E_P$  are defined in the following way.
 
A pair $(Q,b)$, where $Q\in V_Q$ and $b\in V_B$, defines an edge $\be=(Q,b)\in E_B$ exactly when $b\notin Q$ and $\deg(b,G_Q)\ge 3$.

A pair $(Q,p)$, where $Q\in V_Q$ and $p\in V_P$, defines an edge $\be=(Q,p)\in E_P$ exactly when $p\in Q$ and $\deg(p,\Ga_Q)\ge 2$.\\
The edges $\be\in E_B\cup E_P$ have no labels;  they are the terminal edges  of finite paths in $G_T$.\\

A walk $\omega=(\be_1...\be_k...)$ is {\em proper}, if it starts from the vertex $P$, and cannot be extended to a larger walk.
In the following, we will consider only proper walks in $E_T$.

\begin{lemma}\label{ram_b}
    If for some vertex $b \in B$, $\deg(b, \Gamma) \geq 3$ then $b$ is a ramification point of $\hat{\gamma} $.
\end{lemma}
\begin{proof}
    By Proposition \ref{ordim}(3) ${\mathop{\rm Ord} \nolimits}(b, \hga) \ge 3 $. 
\end{proof}
The set of vertices $V_R = \{b \in V_B \cup V_P,\; \deg(b, \Gamma)\geq 3 \}$ will correspond to  the ramification points of $\gamma$ which are the intersection points of at least three copies or the boundary points of the critical set.

 \begin{proposition}\label{fin_cyc}
    Each vertex $Q \in V_Q$ of the graph $G_T$ has an outgoing edge. If $Q$ lies in some cycle, this edge is unique.
 \end{proposition}
 \begin{proof}
  If $Q \in V_Q$, then $\#Q = k \geq 3$. Then two cases are possible:\\
  (i) For any $p \in Q$,\; $\deg(p, \Gamma_{Q}) =1$. 
  Since $Q$ is a cyclic vertex in the graph $G_T$, there exists a vertex $Q'$ that lies in the same cycle as $Q$, such that $\phi_i(Q)=Q'$ and $\#Q' = \#Q= k$.  It follows that in $\Gamma_Q$ there is a white vertex $w_i$ such that $\deg(w_i, \Gamma_Q) = k$.

  As $\Gamma_Q$ is a minimal connected graph containing $Q$ and a white vertex $w_i$ for any vertex $v \neq w_i$,  $\deg(v, \Gamma_Q ) < 3$. Therefore, the vertex $Q$ has at most one outgoing edge.

  Since we arbitrarily chose the vertex $Q$, this is true for any vertex in the cycle.

  If we suppose that there is a vertex $p \in Q$ such that $\deg(p, \Gamma_Q) \geq 2$, then 
  $\max\limits_j(\deg(w_j, \Gamma_Q )) < k$ as the graph $\Gamma_Q$ excludes white vertices of degree 1. However, we know from a previous case that graph $\Gamma_Q$ must have a white vertex of degree $k$, which contradicts our assumption.
  
\end{proof}

In the transformation graph $G_T$, there are 3 types of proper walks. They produce the addresses of the ramification points of $\hga$.\\

{\bf  1.}   $\Omega_Q$  denotes the set of  proper walks whose edges belong only to $E_Q$.

\begin{lemma} \label{inf_walks}
Each walk $\omega=(\be_1,\be_2,...)\in\Omega_Q$ is infinite and terminates in a cycle.
It defines a preperiodic address ${\bm\beta}=\beta_1\beta_2...$, where $\beta_k=\psi(\be_k)$ such that $\pi({\bm\beta})=x\in\hga$, $\Nfi({\bm\beta}) \ge 3$,
and one of the following conditions  holds:\\
1) $x\in \partial K$ and $x$ is a cut point of $\hat \gamma$;\\
2) $x\notin \partial K$ and $x$ is a ramification point of $\hat \gamma$.
\end{lemma}

 \begin{proof}\label{path_bk}
 As graph $G_T$ is finite and by Proposition \ref{fin_cyc} at least one edge comes from each vertex $Q \in V_Q$, there exist $k, l \in \mathbb{N}$ such that $\bm{e}_k = \bm{e}_{k+l}$. It follows that infinite walk $\omega$ terminates in a cycle and for any of the  edges of $\omega$, $\psi({\bm e_j})= {{\beta}}_j$. Then $\omega$ defines a preperiodic address $\bm{\beta}={\beta}_{1}{\beta}_{2} \hdots \overline{{\beta}_{k+1} \hdots {\beta}_{k+l}}$. 

 The point $x = \pi({\bm\beta})$ belongs to  $\hat{\gamma}$ because for any finite subword ${\bm \beta}|_{l}$ of ${\bm{\beta}}$ of length $k$ $\#\phi_{\bm{\beta|_{k}}}(P)\geq 3$. There exist $p_1, p_2 \in \partial K$ such that for $j=1,2\;$
the minimal arc $(p_j, K_{\bm{\beta|_{k}}}) \in \hat{\gamma}$. Let $(p_j, K_{\bm {\beta}}) = \lim\limits_{k \to \infty} (p_j, K_{\bm {\beta|_{k}}})$. The arc $(p_j, K_{\bm {\beta}}) \in \hat{\gamma}$ and contains the point $x$ since ${\bm \beta}$ is an address of $x$. Therefore, $x \in \hat{\gamma}$. \\
  Since $\Nfi({ \bm \beta}) = k \geq 3$   by Corollary \ref{ram_adr}  a point $x \in \hat{\gamma}$ with address $\bm{\beta}$ is a boundary cut-point of $\hat{\gamma}$  or a non-boundary ramification point of $\hat{\gamma}$.
 \end{proof}

{\bf 2.} $\Omega_B$ denotes a set of proper finite paths $\omega_b$ in $E_Q$ that terminate with an edge $\be_b=(Q,b )\in E_B$. They  correspond to ramification points of $\hga$ that are images of  $b $. 

\begin{lemma}\label{path_bk}
    Let $\omega=(\be_1,...,\be_k,\be_b)$ be a path in $\Omega_B$. Let  $j_i=\psi(\be_i)$, $\bj=j_1...j_k$ and ${\bm e}_b = (Q_{i_k}, b_l)$.  
   Then  the point $x = S_\bj(b_l)$ is a ramification point of  $\hat \gamma$ and $\pi^{-1}(x) = {\bm j}\pi^{-1}(b_l)$.

\end{lemma}
 \begin{proof}
 If  $\deg (b_l, \Gamma_{Q_{i_k}}) = l \geq 3$, by Proposition \ref{ordim}(4)\\ ${\mathop{\rm Ord} \nolimits}(S_{\bm j}(b_l), \hat \gamma)\geq {\mathop{\rm Ord} \nolimits}(b_l, \phi^*_{\bm j}(\hat \gamma)) \geq l \geq 3$. Then the point $x = S_{\bm j}(b_l)$ is a ramification point of  $\hat \gamma$. Seeing that $b_l \notin Q_{i_k}$ we get  $x = S_{\bm j}(b_l) \notin \partial K_{{\bm j}}$. Сonsequently, $\pi^{-1}(x) = {\bm j}\pi^{-1}(b_l)$.
 \end{proof}

{\bf 3.} $\Omega_P$ denotes a set of proper finite paths $\omega_b$ in $E_Q$ that terminate with an edge $\be_p=(Q,p )\in E_P$. They  correspond to the boundary cut points or ramification points of $\hat{\gamma}$ that belong to the critical set.

\begin{lemma}\label{path_pk}
    Let $\omega=(\be_1,...,\be_k,\be_p)$ be a path in $\Omega_P$. Let  $j_i=\psi(\be_i)$, $\bj=j_1...j_k$ and ${\bm e}_p = (Q_{i_k}, p_l)$.  
   Let $x = S_\bj(p_l)$. Then\\
   1) If $x\in \partial K$ then $x$ is a cut point of $\hat \gamma$.\\
2) If $x\notin \partial K$ then $x$ is a ramification point of $\hat \gamma$.
\end{lemma}

\begin{proof}
 1) If  $\deg (p_l, \Gamma_{Q_{i_n}}) = l \geq 2$, there will be $l$ arcs $\gamma_1, \dots \gamma_l$. Each $\gamma_i \subset \hat{\gamma}$ connects $x$ and $y_i \in \partial K_{{\bj}}$. Therefore ${\mathop{\rm Ord} \nolimits}(x, \hat \gamma) \geq l \geq 2$.

 2)  If  $\deg (p_l, \Gamma_{Q_{i_n}}) = l \geq 2$, there will be $l$ arcs $\gamma_1, \dots \gamma_l$. Each $\gamma_i \subset \hat{\gamma}$ connects $x$ and $y_i \in \partial K_{\bj}$. Since $x= S_{\bj}(p_l)$,  $x\in \partial K_{\bj}$. It means that $x\notin \partial K$ and $x\in K_{\bj^{-} i_n}$, where $\bj^{-} = j_1\hdots j_{n-1}$ and ${\mathop{\rm Ord} \nolimits}(x, K_{\bj^{-} i_n} \cap \hat{\gamma}) > 0$. Therefore ${\mathop{\rm Ord} \nolimits}(x, \hat \gamma) \geq l+1 \geq 3$.

 \end{proof}

\begin{lemma}
       In the transformation graph $G_T$ \; $\#(\Omega_Q \cup \Omega_B \cup \Omega_P) \le M  -2$, where $M = \#P$.
\end{lemma}\label{fin_walks}
\begin{proof}
    By \ref{fin_cyc}, any two cycles $\sigma_1, \sigma_2$ in $G_T$ are independent. From  Theorem \ref{inf_ram} it follows that the set $ \Omega_Q$ is finite. Since the graph $G_T$ is finite, the number of paths in this graph is finite. Therefore, $\Omega_B$ and $\Omega_P$ are also finite. If for some $w_i \in W$ $\#\phi_i(P) = k\geq 3$ then graph $\Gamma_i$  may contain at most $k-2$ white vertices with a degree greater than or equal to 3.
\end{proof}

Let $I_{\partial K}$ be the set of addresses of the boundary points and $I_{C}$ be the  set of addresses of the critical  points. If for some address $\alpha =\beta \alpha_i,\; \alpha_i \in I_{C}, \beta \in I^* $, $\Nfi(\alpha) \geq 3$ then $x= \pi(\alpha)$ is a point of $\hat{\gamma}$ with multiple addresses.

 If  $ \alpha \in I_{\partial K} $ and $\Nfi(\alpha) \geq 3$ then $x= \pi(\alpha)$ is a boundary cut point of $\hat{\gamma}$.\\

 If the address $\alpha$ does not correspond to a boundary point or a point with multiple addresses, then $x = \pi(\alpha)$ is a point with a single address $\alpha$.

  Consider an example of a $P$-sprout in  Figure \ref{ram_tree}. Using the transformation graph, we will find the addresses of the ramification points.

\begin{figure}[htp]\label{ram_tree}
\centering
\includegraphics[width=.8 \textwidth]{ 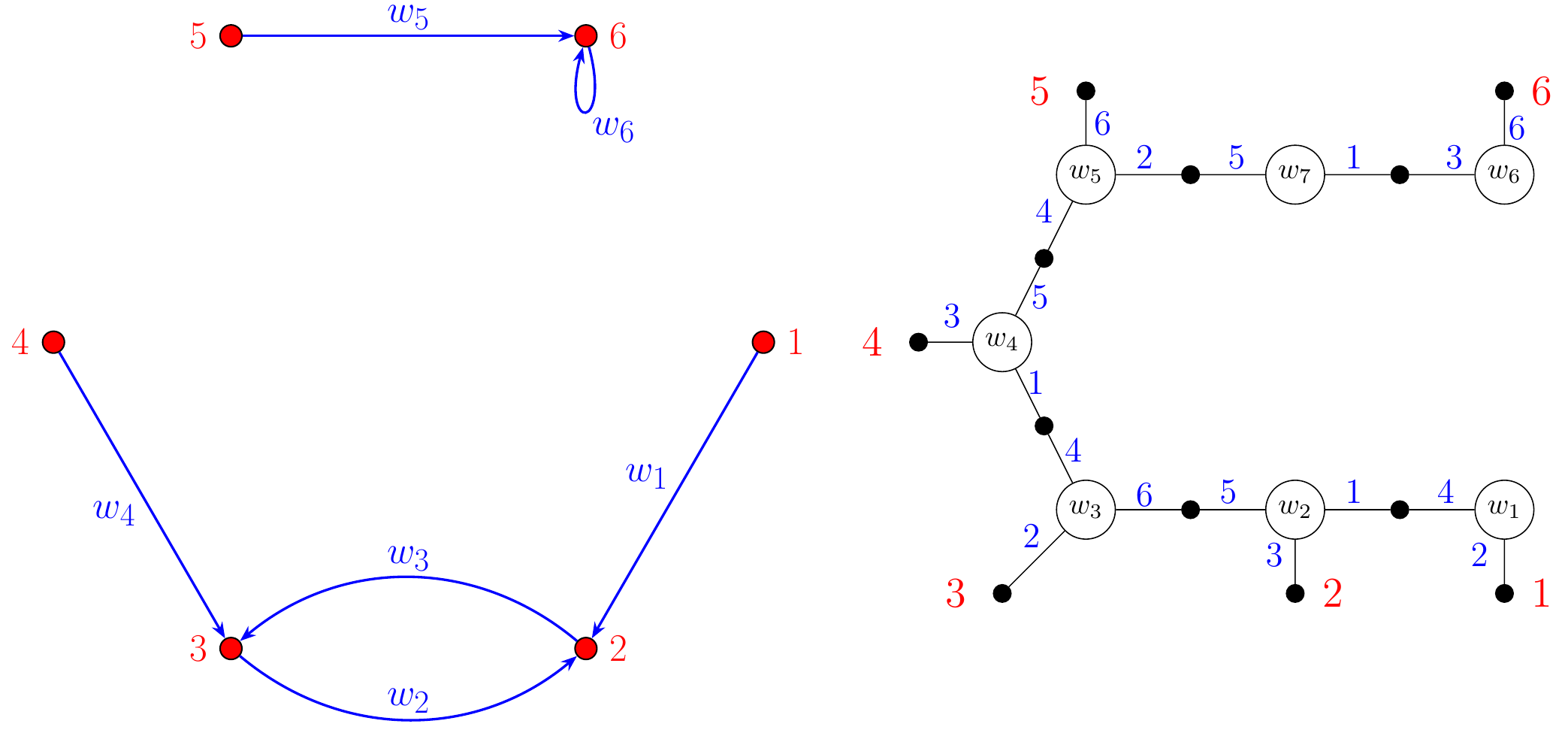}
\caption{ $P$-sprout on the right and its index diagram on the left }
\label{fig:ram_tree}
\end{figure}

\begin{figure}[htp]\label{fig:gt}
\centering
\includegraphics[width=.8 \textwidth]{ 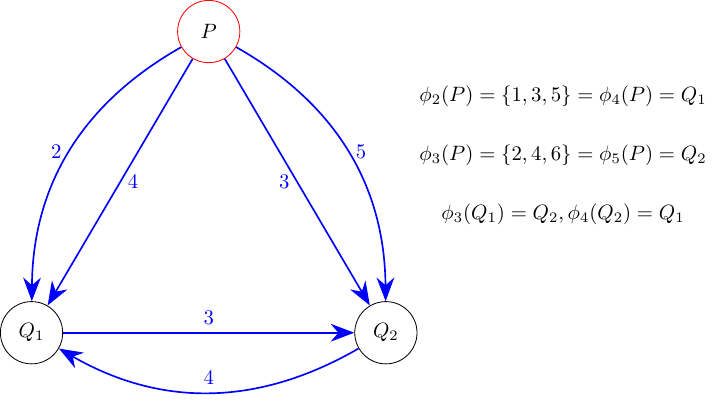}
\caption{The transformation graph for Figure \ref{fig:ram_tree}. We get the address set $ J= \{2\overline{34}, \overline{34},  \overline{43}, 5\overline{43}\}$.}

\end{figure}
The points $a, b, c, d $ are the ramification points of the main tree. \\$\pi^{-1}(a) = \overline{34},\; \pi^{-1}(b) = 2\overline{34},\;   \pi^{-1}(c) = \overline{43}, \;\pi^{-1}(d) = 5\overline{43}$

Get the order of points in the main tree. $\Ord(a, \hat{\gamma}) = N_{\phi}(\overline{34})=3, $  \\$\Ord(b, \hat{\gamma}) = N_{\phi}(2\overline{34})=3, \Ord(c, \hat{\gamma}) = N_{\phi}(\overline{43})=3, \Ord(d, \hat{\gamma}) = N_{\phi}(5\overline{43})=3$.

\begin{figure}[h]\label{fig:ram_ex}
\centering
\includegraphics[width=.6 \textwidth]{ 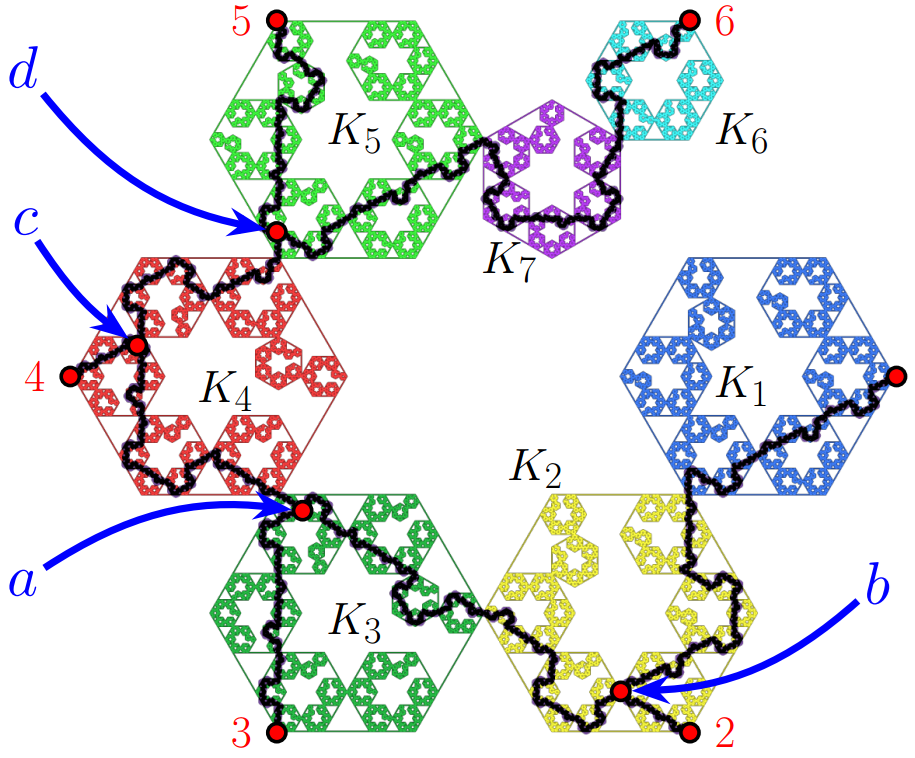}
\caption{An example of a self‑similar dendrite, defined by a $P$-sprout in the Figure \ref{fig:ram_tree}. $\Ord(a, \hat{\gamma}) = N_{\phi}(\overline{34})=3,  $  $\Ord(b, \hat{\gamma}) = N_{\phi}(2\overline{34})=3, \Ord(c, \hat{\gamma}) = N_{\phi}(\overline{43})=3, \Ord(d, \hat{\gamma}) = N_{\phi}(5\overline{43})=3$. }

\end{figure}

\begin{theorem}\label{radr}
    Let $G_T$ be the transformation graph of the $P$-sprout $\Gamma$. Then, for each ramification point $x$ of the main tree $\hat{\gamma}$, which is not contained in $B\cup P$, at least one of the following conditions holds:

1) There is a walk $\omega \in \Omega_Q$ that defines a preperiodic address ${\bm\beta}=\beta_1\beta_2...$,  such that $\pi({\bm\beta})=x$. This walk consists of some path $(\be_1,...,\be_k)\in \Omega_Q$ and a cycle
$(\be_k,...,\be_n)\in \Omega_Q$.

2) There is a path $\omega=(\be_1,...,\be_k,\be_b)\in \Omega_B$ with an end point $b$, which defines a multiindex $\bj=j_1...j_k$, where  $j_i=\psi(\be_i)$, such that  
    $x = S_{\bj}(b)$ .
    
3) There is a path $\omega=(\be_1,...,\be_k,\be_p)\in \Omega_p$ with an endpoint $p\in P$, that defines a multiindex $\bj=j_1...j_k$, where  $j_i=\psi(\be_i)$, such that  
    $x = S_\bj(p)$.

If $x$ is a ramification point of $\hat{\gamma}$ with a unique address, then only  1) is valid.
If $x$ is a ramification point with several addresses, then any of  statements 1) -- 3) may be valid.
\end{theorem}

\begin{proof}  Let $x$ be a ramification point of $\hat\gamma$ and let $\bmb={\beta_1\beta_2....}$ be its address.

Since $x$ belongs to a unique copy $K_{\beta_1}$ of order 1,   there is a copy $K_{\bf j} \ni x$, where ${\bf j} ={\beta_1\hdots \beta_n}$, for which all vertices $Q_i$ on the path 
     $P \overset{\beta_1}{\to} Q_{1} \overset{\beta_2}{\to} Q_{2}  \dots  \overset{\beta_n}{\to} Q_{n}$ in $G_T$    are different. The graph $\Gamma_{Q_n}$ satisfies one of the three following conditions.\\

     1)  There is a vertex $w_{\beta_{n+1}} \in W$ such that $\deg(w_{\beta_{n+1}},\Gamma_{Q_n})\ge 3$ and $x \in K_{{\bf j}\beta_{n+1}}$. Then there is an edge $Q_n\overset{\beta_{n+1}}{\to}Q_{{n+1}}$.
     
      If $Q_{n+1}$ is different from all $Q_i,i=1,...,n$, we repeat the induction step for the copy $K_{\bj'}\ni x$ and $\bj'=\beta_1...\beta_n \beta_{n+1}$.

     If $Q_{n+1}=Q_i$ for some $i\le n$, then the vertices $Q_i,...,Q_n$ form a cycle.  By proposition \ref{fin_cyc}, each of the vertices $Q_k$ of this cycle has a unique outgoing edge in $G_T$, which is defined by a unique $\beta_{k+1}$.    Therefore,
     $Q_{n+1+k}=Q_{i+k}$  and  $\beta_{n+1+k}=\beta_ {i+k}$ for any $k \in\mathbb{N}$.

     Consequently, by the finiteness of the set $V_Q$, each infinite walk in $\Omega_Q$ is preperiodic, which proves 1).
    
    2) There is a vertex $b_k \in B$, $\deg(b_k, \Gamma_{Q_{n}}) \ge 3$. Then there is an edge $Q_{n} \to b_k$ in the transformation graph $G_T$.
By  lemma \ref{path_bk}, there is a ramification point $x = S_{{\bf j}}(b_k)$ such that ${\bf j}$ is the initial word of all addresses of $x$. Therefore, $\pi^{-1}(x) = {\bf j}\pi^{-1}(b_k)$, which proves 2). \\

 3) There is a vertex $p_k \in Q_{n}$, $\deg(p_k, \Gamma_{Q_{n}}) \ge 2$. Then there is an edge $Q_{n} \to p_k$ in the transformation graph $G_T$. By  lemma \ref{path_pk}, there is a ramification point $x = S_{{\bf j}}(p_k)$ such that $\pi^{-1}(x) \supseteq {\bf j}\pi^{-1}(b_k)$.\\

If for a point $x$ condition 2 or 3  is satisfied, $x$ is a point of the critical set or its image and therefore has more than a  unique address. Сondition 1 is necessarily fulfilled, since for the unique address $\alpha$ of $x\;$, $N_{\phi}(\alpha) \geq 3$ and therefore there exists a walk $\omega(P,\dots) \in \Omega_Q$.

\end{proof}

As follows from the above, to find the addresses of all ramification points of  $\hat\gamma$, it is necessary to find all the cycles in the transformation graph $G_T$.
 By proposition \ref{fin_cyc}, all these cycles are independent. Since each of them produces a ramification  point or a boundary cut point, their number is  not greater than $\#P-2$. For each of these cycles one should find the set of all paths in $E_Q$ connecting  $P$ and this cycle.  Thus, we obtain the set $\Omega_Q$.
To find the sets   $\Omega_B$ and $\Omega_P$, 
one should find the set of all the paths in $G_T$ 
that connect $P$ and $b \in V_B$ or $p \in V_P$, respectively.

\section*{Acknowledgments}
The authors especially thank Dmitry Mekhontsev, whose package IFSTile \cite{IFStile} was the main  tool  for the construction  of self-similar and self-affine sets shown in the paper.\\
The study was carried out under the state contract of the Sobolev Institute of Mathematics (project FWNF-2026-0026).

\noindent
Andrei Tetenov \\
Sobolev Institute of Mathematics, \\
Novosibirsk, Russia. \\
e-mail: a.tetenov@gmail.com
\bigskip
\noindent

Ivan Yudin \\
Sobolev Institute of Mathematics, \\
Novosibirsk, Russia. \\
e-mail: uivan566@gmail.com
\bigskip
\noindent

Dmitrii Drozdov \\
Sobolev Institute of Mathematics, \\
Novosibirsk, Russia. \\
e-mail: d.drozdov1@g.nsu.ru
\bigskip
\end{document}